\documentclass[11pt]{amsart}

\usepackage{url}
\usepackage{psfrag}
\usepackage{graphicx}
\usepackage{comment}
\usepackage{array}
\usepackage{amsmath,amsfonts,amssymb,amsxtra,amsthm}
\usepackage{mathrsfs}
\usepackage{lineno}
\usepackage{verbatim}
\usepackage{xy}
\usepackage{times}
\xyoption{all}
\usepackage{stmaryrd}
\usepackage{appendix}
\usepackage{xr}

\DeclareMathOperator{\betti}{b_1}

\DeclareMathOperator{\comp}{Comp}
\DeclareMathOperator{\core}{Core}

\DeclareMathOperator{\depthpc}{depth_{\text{pc}}}

\DeclareMathOperator{\elliptic}{Ell}

\DeclareMathOperator{\img}{Im}

\DeclareMathOperator{\jcomp}{JComp}

\DeclareMathOperator{\Mod}{Mod}

\DeclareMathOperator{\nbhd}{Nbhd}
\DeclareMathOperator{\ninj}{NInj}

\DeclareMathOperator{\rk}{rk}

\DeclareMathOperator{\scott}{sc}

\def\jsj{\textsc{JSJ}}
\def\ajsj{\textsc{aJSJ}}

\def\rjsj{\textsc{rJSJ}}

\def\qh{\textsc{QH}}

\def\hnn{\textsc{HNN}}
\def\csa{\textsc{CSA}}

\def\gad{\textsc{GAD}}

\newcommand{\E}{\mathcal{E}}
\newcommand{\G}{\mathcal{G}}

\newcommand{\V}{\mathcal{V}}
\newcommand{\W}{\mathcal{W}}

\newcommand{\HH}{\mathcal{H}}

\def\sl2c{\ensuremath{{SL}(2,\mathbb{C})}}
\def\t1sl2c{{\mathfrak sl}_2\mathbb{C}}
\def\free{\ensuremath{\mathbb{F}}}
\def\zee{\mathbb{Z}}

\def\define{\raisebox{0.3pt}{\ensuremath{:}}\negthinspace\negthinspace=}

\def\immerses{\looparrowright}
\def\onto{\twoheadrightarrow}
\def\into{\hookrightarrow}

\newcommand{\adjoin}[2]{\ensuremath{#1\!\left[#2\right]}}
\newcommand{\adjoinroot}[3]{\ensuremath{\adjoin{#1}{\sqrt[#3]{#2}}}}
\newcommand{\grushko}[3]{\ensuremath{#1_1*\cdots*#1_{#2}*\free_{#3}}}

\def\cent{Z}

\makeatletter
\newcommand{\Doubletwo}[3]{ \left\{ #1 \mid #3\right\} }
\newcommand{\Doubleone}[1]{ \left\{ #1 \right\} }
\newcommand{\Doublethr}[3]{ \left\{ #1 \right\}_{#3} }
\newcommand{\set}[1]{%
\@ifnextchar:{\Doubletwo{#1}}{\@ifnextchar_{\Doublethr{#1}}{\Doubleone{#1}}}%
}
\makeatother

\makeatletter
\newcommand{\grouptwo}[3]{ \langle #1 \mid #3\rangle }
\newcommand{\groupone}[1]{ \langle #1 \rangle }
\newcommand{\group}[1]{%
\@ifnextchar:{\grouptwo{#1}}{\groupone{#1}}%
}
\makeatother

\newcommand{\term}[1]{{\emph{#1}}}

\newcommand{\mobius}{M\"obius}

\newtheorem{theorem}{Theorem}[section]
\newtheorem{lemma}[theorem]{Lemma}
\newtheorem{example}[theorem]{Example}

\theoremstyle{definition}
\newtheorem{definition}[theorem]{Definition}

\theoremstyle{remark}
\newtheorem{remark}[theorem]{Remark}

\newcommand{\mnote}[1]{}

\title[Adjoining roots]{Krull dimension for limit groups IV:\\
Adjoining roots} 

\author[Lars Louder]{Larsen Louder}
\address{Department of Mathematics\\
University of Michigan \\
Ann Arbor, MI 48109-1043\\
USA}
\email[Larsen Louder]{llouder@umich.edu, lars@d503.net}



\thanks{Most of this research was done while at the University of
  Utah. The author also gratefully acknowledges support from the
  National Science Foundation, MSRI, and Rutgers University, Newark.}

\begin{document}

\begin{abstract}
  This is the fourth and last paper in a sequence on Krull dimension
  for limit groups, answering a question of Z.~Sela. In it we finish
  the proof, analyzing limit groups obtained from other limit groups
  by adjoining roots. We generalize our work on Scott complexity and
  adjoining roots from the previous paper in the sequence to the
  category of limit groups.
\end{abstract}

\maketitle


\section{Introduction, Notation, Theorems}
\label{sec:introduction}

\par It will take a moment to establish the notation and define the
objects needed to state our main theorem. Roughly, we are interested
in solutions, in the category of limit groups, to equations of the
form ``adjoin a root to $g$.'' We can give no specific
characterizations of solutions, but under special circumstances
arising in the second paper in this series, \cite{louder::stable}, we
are able to show that most of the time solutions are unique.

\par The notation $\cent_G(E)$ indicates the centralizer in $G$ of a
subgroup $E$.  The set of images of edge groups incident to a vertex
group $V$ of a graph of groups decomposition is denoted by
$\E(V)$. The phrase ```$X$' is controlled by `$Y$''' should
be read as ``there is a function $f,$ defined independently of `$X$'
and `$Y$', such that $X\leq f(Y)$''.

\par Let $G$ be a group.  A system of equations over $G$ is a
collection of words in the alphabet $\set{x_i,g}:{g\in G},$ where
the $x_i$ are variables distinct from the elements of $G$.  The
elements of $G$ are the coefficients, and the coefficients occuring in
$\Sigma$ are the coefficients of $\Sigma$.  If $\Sigma$ is a system of
equations over $G$ there is a canonical group $G_{\Sigma}$ associated
to $\Sigma$ with the presentation $\group{x_i,G}:{\Sigma},$
where the $x_i$ are the variables occuring in $\Sigma$.  If the map
$G\to G_{\Sigma}$ is injective then $\Sigma$ has a solution.  If $G<H$
and the inclusion map extends to $G_{\Sigma}$ then $\Sigma$ has a
solution in $H$.  In analogy with field extensions, suppose $\Sigma$
is a system of equations over $G$.  If $G<H$ and the inclusion map
extends to a surjection $G_{\Sigma}\onto H$ then $H$ is a
\term{splitting group} for $\Sigma,$ and $G$ is the \term{ground
  group}.  Splitting groups are partially ordered by the relation
``maps onto.''  Every pair $G<H$ is a ground-splitting pair for some
(in general, many) system of equations $\Sigma(G,H)$.  A tuple
$(G,H,G')$ is \term{flight} if $H$ and $G'$ are both splitting groups
over $G,$ and $H\onto G'$.

\par One may ask for splitting groups in a category $\mathcal{C}$ of
groups. If $H\in\mathcal{C}$ is a splitting group, then $H$ is a
splitting group in $\mathcal{C}$. If $\mathcal{C}$ is the class of all
groups, then there are maximal $\mathcal{C}$--splitting groups, but
this is not the case for general classes.

\par A sequence of inclusions $\G=(\G(0)<\G(1)<\dotsb)$ is a
\term{tower}.  A \term{staircase} is a pair of sequences $(\G,\HH)$
such that $\G$ is a tower and $\G(i)$ is a splitting group for (some
system) $\Sigma(\G(i-1),\HH(i)),$ that is, $(\G(i-1),\HH(i),\G(i))$ is a
flight.  All staircases considered in this paper have the property
that all coefficients lie in $\G(0)$.  The name staircase comes from
the fact that a commutative diagram representing one looks like a
staircase and walks up a tower.

\begin{definition}[Adjoining roots]
  \label{def:adjunctionofroots}
  Let $G$ be a finitely generated group, $\E$ a collection of
  nontrivial abelian subgroups of $G$.  For each $E\in\E,$
  let $\mathcal{F}(E)$ be a collection of finite index supergroups of
  $E,$ with an inclusion map $i_{E,F}\colon E\into F$ for each
  $F\in\mathcal{F}(E),$ and let $\mathcal{F}(\E)$ be the
  collection $\set{\mathcal{F}(E)}$.  Let
  \[
  \adjoinroot{G}{\E}{\mathcal{F}(\E)} \define 
  \group{G,F}:{E=i_{E,F}(E)}_{F\in\mathcal{F}(E),E\in\E}
  \]

  A finitely generated group $H$ is obtained from $G$ by
  \emph{adjoining roots $\mathcal{F}(\E)$ to $\E$}
  if $G<H$ and the inclusion map extends to a surjection
  \[
    \adjoinroot{G}{\E}{\mathcal{F}(\E)}\onto H
    \] 
  Let $\Sigma=\Sigma(\E,\mathcal{F}(\E))$ be a
  system of equations corresponding to the identification of $E$ with
  $i_{E,F}(E)$ for all $E$ and $F\in\mathcal{F}(E)$.  Then $H$ is a
  splitting group for $\Sigma$.  We call $H$ a cyclic extension of $G$
  because the relations are all of the form ``adjoin a root to $G$.''

  Most of the time the specific nature of $\mathcal{F}$ is immaterial,
  and we usually eliminate it from the notation.  To further compress
  the language used, sometimes we simply write that $H$ is obtained
  from $G$ by adjoining roots.
\end{definition}

\par A group is \term{conjugately separated abelian}, or \csa, if
maximal abelian subgroups are malnormal.  Let $\sim_Z$ be the relation
``is conjugate into the centralizer of''. This is an equivalence
relation as long as the group is \csa.  Two important consequences of
\csa\ are commutative transitivity and that every nontrivial abelian
subgroup is contained in a unique maximal abelian subgroup.

\par Commutative transitivity can occasionally be used to simplify
systems of equations. Suppose $H$ is obtained from $G$ by adjoining
roots $\mathcal{F}(\E)$ to $\E$.  Let $\eta$ be the
inclusion map.  We remove some redundancy by singling out a
subcollection of each of $\E$ and $\mathcal{F}(\E),$
and replacing each subcollection by a single element.  Fix some
$\sim_Z$ equivalence class $[E]$.  By conjugating we may assume that
each element of $[E]$ is a subgroup of $\cent_G([E])$.  Replace $[E]$ by
$\set{\cent_G([E])},$ and replace $\cup_{B\in[E]}\mathcal{F}(B)$ by
\[
\group{\cent_G([E]),F}:{B=i_{B,F}(B)}_{B\in[E],F\in\mathcal{F}(B)}^{\text{ab}}
\]
Then by commutative transitivity $H$ is a quotient of
\[
\adjoinroot{G}{\E}{\mathcal{F}(\E)}
\] 
Since limit groups are \csa\ we make this reduction without comment.
Since $\mathcal{F}(E)$ has a single element after this simplification,
we will generally use the less ostentatious notation $F(E)$ or just
$\sqrt{E}$. We will call a system of equations without any such
redundancy \term{reduced}.

\begin{definition}[Staircase]
  \label{def:seqadjunctionofroots}
  A \term{cyclic staircase} is a staircase, with tower $\G,$ equipped
  with a family families $\E$ of subgroups $\E_i$ of
  $\G(i),$ $(\G,\HH,\E),$ such that 
  \begin{itemize}
  \item $(\G(i-1),\HH(i),\G(i))$ is a flight; $\HH(i)$ is obtained from
    $\G(i-1)$ by adjoining roots to $\E_{i-1}$
  \item Each $E'\in\E_{i}$ in $\G(i)$ centralizes, up to
    conjugacy, the image of an element $E$ of $\E_{i-1}$.  If
    $E\in\E_{i-1}$ is mapped to $E'\in\E_{i}$ then
    we require that the image of $\cent_G(E)$ in $\cent_{G'}(E')$ be finite
    index.
  \end{itemize}

  To fix notation, the maps $\G(i)\into\G(i+1),$ $\G(i)\into\HH(i+1),$
  and $\HH(i+1)\onto\G(i+1)$ are denoted by $\eta_i,$ $\nu_i,$ and
  $\pi_{i+1},$ respectively.  The length of $\G$ is denoted
  $\Vert\G\Vert$.

\end{definition}

\par It will be handy to have a rough description of a staircase.  A
staircase of limit groups is
\begin{itemize}
\item \term{freely decomposable} if all $\G(i)$ are freely
  decomposable
\item \term{freely indecomposable} if all $\G(i)$ are freely
  indecomposable
\item \term{\qh--free} if no $\G(i)$ has a \qh\ subgroup
\item \term{mixed} if it has both freely decomposable and freely
  indecomposable groups, or, if freely indecomposable, has both groups
  with and without \qh\ subgroups. Otherwise it is \term{pure}.
\end{itemize}

\begin{definition}
  Let $(i_j)$ strictly increasing sequence of indices.  A staircase
  $(\mathcal{V},\mathcal{W}),$ such that $\mathcal{V}(j)=\G(i_j)$ and
  $\mathcal{W}(j)=\HH(i_j),$ with maps obtained by composing maps from
  $(\G,\HH),$ is a \term{contraction} of $(\G,\HH),$ and is \term{based
    on} $(i_j)$.
\end{definition}

\par To see that a contraction of a cyclic staircase is a staircase
consider the following diagram:
\begin{figure}[h]

\centerline{%
  \xymatrix{%
            & \HH(i_j+1)\ar@{->>}[d] & \cdots & \HH(i_{j+1}-1)\ar@{->>}[d] & \HH(i_{j+1})\ar@{->>}[d] \\
    \G(i_j)\ar@{^(->}[r]\ar@{^(->}[ur] & \G(i_j+1) & \cdots & \G(i_{j+1}-1)\ar@{^(->}[r]\ar@{^(->}[ur]  & \G(i_{j+1})
  }
}
\end{figure}

\par Each $E\in\E_i$ has finite index image in its
counterpart in $\E_{i+1},$ the image of $E$ in its
counterpart in $\E_{i_{j+1}}$ is finite index.  Extending an
abelian group by a finite index super-group multiple times can be
accomplished by extending once.

\par The need for contractions explains the restriction that each
$E\in\E_i$ contain a conjugate of the image of some
$E'\in\E_{i-1}$.  If this is not the case, then there is no
hope for the existence of contractions; we can't adjoin a root to an
element that isn't there.

\par A \term{segment} of a staircase is a contraction whose indices
are consecutive, that is $i_{j+1}-i_j=1$ for all $j$.

\par Let $\E$ be a collection of elements of a \csa\ group
$G$. We denote by $\Vert\E\Vert$ the number of $\sim_Z$
equivalence classes in $\E$.  The \term{complexity} of
$(\G,\HH,\E)$ is the triple
$\comp((\G,\HH,\E))\define(\betti(\G),\depthpc(\HH),\Vert\E\Vert)$.
Complexities are not compared lexicographically: $(b',d',e')\leq
(b,d,e)$ if $b'\leq b,$ $d'\leq d,$ and $e'\leq e+2(d-d')b$.  That
this defines a partial order follows easily from the definition.  The
inequality is strict if one of the coordinate inequalities is strict.
See Definition~\ref{def:depth} and the material thereafter for a
discussion of depth.  Another immediate consequence of the definition
of $\leq$ is that it is locally finite.\footnote{Fix $a$ and $b$.
  Then $\set{x}:{a\leq x\leq b}$ is finite.}

\par Let $(\G,\HH,\E)$ be a staircase.  The quantity
$\ninj((\G,\HH,\E))$ is the number of indices $i$ such that
$\HH(i)\onto\G(i)$ is \emph{not} an isomorphism.

\begin{theorem}
  \label{maintheorem}
  Let $(\G,\HH,\E)$ be a staircase.  There is a function
  $\ninj(\comp((\G,\HH,\E)))$ such
  that \[\ninj((\G,\HH,\E))\leq\ninj(\comp((\G,\HH,\E)))\]
\end{theorem}

\begin{remark}
  Although it would be nice to assign a complexity $c()$ to a limit
  group such that if, in a flight $(G,H,G'),$ $c(G)=c(G'),$ then
  $H\onto G'$ is an isomorphism, this doesn't seem possible, and the
  approach taken in this paper requires that complexities be computed
  and compared in context.
\end{remark}

\subsection*{Acknowledgments}

\par The author thanks Mladen Bestvina, Mark Feighn, and Zlil Sela for
many discussions related to this paper.

\section{Complexities of sequences}

\par The main object which enables this analysis of adjoining roots is
the \jsj\ decomposition, a device for encoding families of splittings
of groups.  This exposition borrows from~\cite{bf::lg,sela::jsj}.  A
\gad, or \term{generalized abelian decomposition} of a group $G$ is a
finite graph of groups decomposition over abelian edge groups such
that every vertex group is marked as one of \term{rigid},
\term{abelian}, or \term{\qh}, where by \qh\ we mean is the
fundamental group of a compact surface with boundary possessing two
intersecting essential simple closed curves.  Moreover, edge groups
adjacent to a \qh\ vertex group must be conjugate to boundary
components of the surface.  If $A$ is an abelian vertex group, the
\term{peripheral} subgroup of $A$ is the subgroup of $A$ which dies
under every map $A\to\zee$ killing all incident edge groups.

\par We say that two \gad's of a limit group are \term{equivalent} if
they have the same elliptic subgroups.  A splitting is \term{visible}
in a \gad\ $\Delta$ if it corresponds to cutting a \qh\ vertex group
along a simple closed curve, a one-edged splitting of an abelian
vertex group in which the peripheral subgroup is elliptic, or is a one
edged splitting corresponding to an edge from an equivalent
decomposition.  If $\Delta$ is a \gad, then $g\in G$ is
$\Delta$--elliptic if elliptic in every one-edged splitting of $G$
visible in $\Delta$.  Let $\elliptic(\Delta)$ be the set of
$\Delta$--elliptic elements.

\par Let $G$ be a freely indecomposable finitely generated group, and
let $\mathcal{C}$ be a family of one-edged splittings of $G$ such that
\begin{itemize}
  \item edge groups are abelian,
  \item noncyclic abelian subgroups are elliptic.
\end{itemize}

\par The main construction of \jsj\ theory is that given a family of
splittings $\mathcal{C}$ satisfying these conditions, there is a \gad\
$\Delta$ such that
$\elliptic(\Delta)=\cap_{C\in\mathcal{C}}\elliptic(C)$.

\par An abelian \jsj\ decomposition of $G$ is a \gad\ $\ajsj(G)$ such
that the set of {\ajsj--elliptic} elements corresponds to the
collection of all one-edged splittings satisfying the bullets above.
The existence of a \jsj\ decomposition is somewhat subtle, as one
needs to bound the size of a \gad\ arising in this
way~\cite[Theorem~3.9]{sela::dgog1}. If $G$ is a nonelementary freely
indecomposable limit group then $G$ has a nontrivial \jsj\
decomposition. If $G$ is elementary, the \jsj\ is a point.

\par In this paper we are interested in the principle cyclic \jsj\
decomposition, which is the \jsj\ associated to the family of
principle cyclic splittings.

\begin{definition}[\cite{sela::dgog1}]
  \label{def:principlecyclicsplitting}
  A one-edged splitting over a cyclic subgroup is \term{inessential}
  if at least one vertex group is cyclic, and is \term{essential}
  otherwise.  A \term{principle cyclic} splitting of a limit group is
  an essential one-edged splitting $G\cong A*_CB$ or $G\cong A*_C,$
  over a cyclic subgroup $C,$ such that either $\cent_G(C)$ is cyclic or
  $A$ is abelian.

  \par The \term{principle cyclic \jsj} of a freely indecomposable
  limit group is the \jsj\ decomposition corresponding to the family
  of principle cyclic splittings.  We denote the principle cyclic
  \jsj\ by $\jsj(G)$.

  \par Let $\E\subset G$.  The \term{principle cyclic \jsj\ of
    $G,$ relative to $\E$} is a \jsj\ decomposition
  corresponding to the family of all principle cyclic splittings of
  $G$ such that each member of $\E$ is elliptic.  We denote
  the relative \jsj\ by $\jsj(G;\E)$.  A \term{principle
    cyclic decomposition} is simply a relative principle cyclic
  \jsj\ for some collection $\E$.
  
  \par The \term{restricted principle cyclic \jsj}, or
  \term{restricted \jsj} for short, of a freely indecomposable limit
  group $G$ \emph{with} \qh\ subgroups is the relative
  \jsj\ decomposition associated to the set of principle cyclic
  splittings whose edge groups are hyperbolic in some other principle
  cyclic splitting.  It is obtained from the \jsj\ by collapsing all
  edges not adjacent to some \qh\ vertex group.  If $G$ doesn't have
  \qh\ vertex groups, then the restricted \jsj\ is just the principle
  cyclic \jsj.  The restricted principle cyclic \jsj\ is denoted by
  $\rjsj(G)$
\end{definition}

\par That limit groups have principle cyclic splittings
is~\cite[Theorem~3.2]{sela::dgog1}.  It need not be the case that
every splitting visible in the principle cyclic \jsj\ is principle;
for instance, a boundary component of a \qh\ vertex group may be the
only edge attached to a cyclic vertex group.  The splitting
corresponding to the boundary component is not essential, but is
certainly visible in the principle cyclic \jsj.

\par In this paper we work primarily with the principle cyclic
\jsj\ of $G,$ indicated by $\jsj(G),$ and the \rjsj.  If $\Delta$ is a
graph of groups decomposition then $T_{\Delta}$ is the Bass-Serre
tree corresponding to $\Delta$.

\par We can give a more explicit description of the principle cyclic
\jsj.  Consider the abelian \jsj\ of a limit group $G$.  Clearly all
\qh\ vertex groups of $\ajsj(G)$ appear as vertex groups of $\jsj(G)$.
If $A$ is an abelian vertex group of $\ajsj(G)$ with noncyclic
peripheral subgroup, since there is no principle cyclic splitting of
$G$ over a subgroup of $A,$ the subgroup of $G$ generated by $A$ and
conjugates of rigid vertex groups having nontrivial intersection with
$A$ must be elliptic in $\jsj(G)$.  If $R$ is a rigid vertex group of
$G$ and an edge group $E$ incident to $R$ has noncyclic centralizer in
$R,$ then the subgroup of $G$ generated by $R$ and any conjugate of a
rigid vertex group $R'$ intersecting $R$ in a nontrivial subgroup of
$E$ is also elliptic.  From this we see that the principle cyclic
\jsj\ of $G$ must have the following form:
\begin{itemize}
  \item Every abelian vertex group has cyclic peripheral subgroup.  If
    $R$ is adjacent to an abelian vertex group $A,$ $E$ the edge
    group, then $R$ does not have an essential one-edged splitting
    over $E$ in which each element of $\E(R)$ is elliptic.
  \item If an edge $e$ incident to a rigid vertex group $R$ has
    noncyclic centralizer in $R,$ then the edge is attached to a
    boundary component of a \qh\ vertex group.
  \item If two edges incident to a rigid vertex group $R$ have the
    same centralizer in $R,$ then they are both incident to
    \qh\ vertex groups, and their centralizer in $R$ is noncyclic.
\end{itemize}

The \jsj\ decomposition of a limit group, be it abelian or principle
cyclic, is only unique up to morphisms of graphs of groups preserving
elliptic subgroups. Some principle cyclic \jsj's are more convenient
to work with than others, and we assume throughout that
\begin{itemize}
  \item Edge groups not adjacent to \qh\ vertex groups are closed
    under taking roots, and edge maps of edge groups into \qh\ vertex
    groups are isomorphisms with the corresponding boundary components.
  \item There are no inessential splittings visible in the \jsj, other
    than from valence one cyclic vertex groups attached to boundary
    components of \qh\ vertex groups.
\end{itemize}

\par Let $R$ be a rigid vertex group of the full abelian \jsj\ of a
limit group $G,$ and let $\bar{R}$ be the subgroup of $G$ generated by
$R$ and all elements with powers in $R$.  A decomposition with the
properties above can be thought of as the \jsj\ decomposition
associated to the family of principle cyclic splittings in which all
$\bar{R},$ $R$ a vertex group of the abelian \jsj, are elliptic.

\par In general, there are infinitely many principle cyclic
decompositions of a limit group, all obtained from the principle
cyclic \jsj\ by folding, cutting \qh\ vertex groups along simple
closed curves, and collapsing subgraphs.

\begin{lemma}
  \label{lem:numberofdecompositions}
  Let $G$ be a limit group and $\E\subset G$ a collection of
  elements of $G$.  Then there are at most $2^{\Vert\E\Vert}$
  equivalence classes of principle cyclic decompositions in which some
  elements of $\E$ are elliptic.
\end{lemma}

\begin{proof}
  If $E\in\E$ is elliptic, then so is any $E'\in\E$
  such that $E\sim_Z E'$.
\end{proof}

\par We need to adapt the definition of the analysis lattice of a
limit group given in \cite[\S4]{sela::dgog1} to the inductive proof
given in section~\ref{finishargument}.  A limit group is
\term{elementary} if it is abelian, free, or the fundamental group of
a closed surface.

\begin{definition}[Principle cyclic analysis lattice]
  The \term{principal cyclic analysis lattice} of a limit group $G$ is
  the rooted tree of groups whose levels are defined as follows:
  \begin{itemize}
    \item[$0$:] $G$
    \item[$\frac{1}{2}$:] The free factors of a Grushko
      decomposition of $G$.
    \item[$1$:] The vertex groups at level $1$ are the vertex
      groups of the \rjsj.  
    \item[$n\left(\frac{1}{2}\right)$:] Rinse and repeat,
      incrementing the index by one each time.
    \end{itemize}
  If an elementary limit group is encountered, it is a terminal leaf
  of the tree.
\end{definition}

\begin{definition}
  \label{def:depth}
  The \term{depth} of a limit group $H$ is the number of levels in its
  principle cyclic analysis lattice, and is denoted $\depthpc(H)$.

  The \term{depth} of a staircase $(\G,\HH,\E)$ is
  $\max\set{\depthpc(\HH(i))},$ and is is denoted
  $\depthpc((\G,\HH,\E))$.  The \term{first betti number} of
  $(\G,\HH,\E)$ is the first betti number of $\G(1)$.
\end{definition}

\par It is not always necessary to refer to the family $\E,$
so we suppress it from the notation when its size is irrelevant.  That
the depth is well defined is a consequence of
Theorem~\ref{analysislatticebound}.

\begin{theorem}
  \label{analysislatticebound}
  The depth of the principle cyclic analysis lattice of a limit group $L$
  is controlled by its rank.
\end{theorem}

\begin{proof}
  We only need to worry about the possibility that the principle
  cyclic analysis lattice contains a long branch of the form
  $G_0>G_1>\dotsb,$ where each $G_i$ is freely indecomposable, has no
  \qh\ vertex groups, no noncyclic abelian vertex groups, and
  $\jsj(G_i)$ has only one nonabelian vertex group $G_{i+1}$.  After
  observing \cite{louder::strict,houcine-2008} that $L$ has a strict
  resolution of length at most $6\rk(L),$ the proof is identical
  to~\cite[Theorem~\ref{STABLE-lem:depthbound}]{louder::stable}.
\end{proof}

\par We motivate our proof of Theorem~\ref{maintheorem} and the
previous definition with an example.

\begin{example}
  Suppose that $\G(i)$ has a one-edged \jsj\ decomposition with two
  nonabelian vertices for all $i$.  Since a limit group has a
  principle cyclic splitting, the one-edged splitting of $\G(i)$ must
  be of the form $\G_1(i)*_{\group{e_i}}\G_2(i)$.  By
  Lemma~\ref{actshyperbolically}, if $H$ is obtained from $G$ by
  adjoining roots, then $G$ acts hyperbolically in all splittings of
  $H$. In particular, every vertex group of $H$ contains a vertex
  group of $G$. If the \jsj\ of $H$ was a loop, then the map to the
  underlying graph kills $G,$ but since the map $G\to H$ is almost
  onto on homology, this cannot happen. Thus each $\HH(i)$ has a
  one-edged \jsj\ decomposition $\HH_1(i)*_{\group{f_i}}\HH_2(i)$.

  The triple $\G(i-1)\into\HH(i)\onto\G(i)$ has the following form: The
  pairs $(\G_j(i-1),\group{e_{i-1}})$ map to the pairs
  $(\G_j(i),\group{e_{i}})$ and $(\HH_j(i),\group{f_i}),$
  and the maps $\eta_i,$ $\nu_i,$ and $\pi_i$ respect these one-edged
  splittings.  We'll show later that in fact $\HH_j(i)$ is obtained
  from $\G_j(i-1)$ by adjoining roots.  By work
  from~\cite{louder::stable} $\G_j(i)$ is obtained from the image of
  $\HH_j(i)$ by iteratively adjoining roots to the incident edge group
  (See Appendix~\ref{appendix_strict}).  Let $\G'_j(i)$ be the image
  of $\HH_j(i)$ in $\G_j(i)$.  Now consider the staircase
  $(\G'_j(i),\HH_j(i))$.  The sequence $\HH_j$ has strictly lower depth
  than $\HH$.

  By induction on $\comp,$ there is an upper bound on the number of
  indices such that $\HH_j(i)\to\G(i)$ is not injective, and for at
  most twice that bound, both maps $\HH_j(i)\to\G_j(i)$ are injective.
  Then $\HH(i)\onto\G(i)$ is strict for such indices.  Since every Dehn
  twist in $\group{f_i}$ pushes forward to a Dehn twist in
  $\group{e_{i+1}},$ $\pi_i$ is an isomorphism.
\end{example}

\section{Aligning \jsj\ decompositions}


\par Let $G$ be a finitely generated group acting on a simplicial tree
$T$ minimally and without inversions.  It is a standard fact that the
quotient $T/G$ is the underlying graph of a graph of groups
decomposition of $G$.  If $H<G$ is a finitely generated subgroup, there
is a minimal subtree $S\subset T$ fixed (setwise) by $H,$ and the
action of $H$ on $S$ endows $H$ with a graph of groups
decomposition.  Additionally, there is an induced map of quotient
graphs $S/H\to T/G$.

\par We are interested in the following problem: Suppose $G,$ $H,$
$T,$ and $S$ are as above, $G$ and $H$ freely indecomposable limit
groups, $T$ the Bass-Serre tree corresponding to the principle cyclic
\jsj\ of $H$.  We say that $G$ and $H$ are aligned if $S/H\to T/G$ is
an isomorphism of graphs and $S/H$ is the underlying graph of the
principle cyclic \jsj\ of $H$. Give a simple computable criterion
which guarantees that $G$ and $H$ are aligned.

\par As long as $G^{ab}\to H^{ab}$ is virtually onto, we are able to
answer this question in a reasonable way, constructing a
(monotonically decreasing) complexity, equality of which will
guarantee alignment of \jsj s.  The properties of the alignment are
then used to construct graphs of spaces and maps between them which
resemble Stallings' immersions. The main idea of this section is that
an inclusion as above must either ``tighten up'' the Grushko/\jsj\,
becoming simpler in a quantifiable way, or can be written as a map of
graphs of groups respecting the \jsj\ decompositions.

\par Let $T$ be the Bass-Serre tree corresponding to the principle
cyclic \jsj\ of $H,$ let $S$ be the minimal subtree for $G$.  The
quotient $S/G$ is finite and it follows from the definitions that the
induced graphs of groups decomposition of $G$ is principle.  For
convenience, we usually conflate underlying graphs and graphs of
groups decompositions.  Let $\Delta_H=T/H$ and $\Delta_G=S/G$ be the
underlying graphs, and let $\eta_{\#}$ be the induced map.  We label a
vertex $v$ of $\Delta_G$ by the corresponding label on $\eta_{\#}(v),$
unless $G_v$ is abelian, in which case we label it abelian anyway.
The map $\eta_{\#}$ is well behaved:

\begin{itemize}
\item If $v$ is rigid then the edge groups adjacent to $v$ have
  nonconjugate centralizers in $G_v$ unless they are all attached to
  boundary components of \qh\ vertex groups.\mnote{careful}
\item Let $B$ be a maximal connected subgraph of $\Delta_G$ such that
  every vertex is abelian.  Commutative transitivity implies that
  $G_B$ is abelian, and the fact that all noncyclic abelian subgroups
  of $H$ are elliptic in $T_H$ implies that $B$ is a tree. 
\item If $v$ is abelian and $\eta_{\#}(v)$ is nonabelian, then
  $\eta_{\#}(v)$ is rigid.
\item A valence one cyclic $v$ is adjacent to a \qh\ $w$.  This follows
  from the assumption that the only edge groups of $H$ allowed to be
  not closed under taking roots are adjacent to \qh\ vertices.
\end{itemize}


\begin{lemma}
  \label{actshyperbolically}
  Let $\eta\colon G\to H$ be a homomorphism of freely indecomposable
  limit groups such that $H^1(H,\zee)\to H^1(G,\zee)$ is injective.
  Then $G$ is hyperbolic in every essential one-edged abelian
  splitting of $H$.

  If $R$ is a nonabelian vertex group of a \gad\ $\Delta_H$ of $H,$
  then $G$ intersects a conjugate of $R$ in a nonabelian subgroup. If
  $\Delta_G$ is the induced decomposition of $G,$ and there is only
  one nonabelian vertex group $R'$ of $\Delta_G$ mapping to $R,$ then
  the map on underlying graphs is a submersion at $R'$.
\end{lemma}

\begin{proof}
  Claim: If $G$ acts elliptically in some essential one-edged
  splitting then there is a map $H\onto\zee$ which kills $G$.  If the
  one-edged splitting is an \hnn\ extension the claim is clear.  If
  not, then both vertex groups of the amalgam have a map onto $\zee$
  which kills the incident edge group.

  To see the second half, suppose not, and let $\Delta'_H$ be the
  decomposition of $H$ obtained by conjugating edge maps to $R$ so
  that all incident edges either have the same or nonconjugate
  centralizers, and folding together edges of the conjugated
  decomposition which have the centralizers. Then pull all the
  centralizers of incident edge groups across the edge they
  centralize. If $T$ is the tree for $\Delta'_H,$ and $S$ is the
  minimal $G$--invariant subtree, then the map $S/G\to T/H$ clearly
  misses the vertex corresponding to $R$. Let $\Delta''$ be the
  decomposition of $H$ obtained by collapsing all edges not adjacent
  to $R$. Then $G$ is elliptic in $\Delta''$. The first part provides
  a contradiction.

  If the map is not a submersion on the level of $\Delta_H,$ then the
  of graphs of groups $\Delta'_G\to\Delta'_H$ is not a submersion onto
  $R$ either, and there is an edge incident to $R$ missed by
  $\Delta'_G$. This edge represents an essential splitting of $H,$ and
  so we again have a contradiction.
\end{proof}


\begin{definition}[Complexity of \jsj s]
  \label{def:compjsj}
  Let $G$ be a finitely generated freely indecomposable limit group
  with principle cyclic decomposition
  $G=\Delta(\mathcal{R},\mathcal{Q},\mathcal{A},\E),$ where
  each $R\in\mathcal{R}$ is rigid, each $Q\in\mathcal{Q}$ is \qh, each
  $A\in\mathcal{A}$ is finitely generated abelian, and each
  $E\in\E$ is an infinite cyclic edge group.  Let
  \begin{itemize}
  \item $c_q(G)\define \vert\sum_{Q\in\mathcal{Q}}\chi(Q)\vert$ is the
    total Euler characteristic of \qh\ subgroups.
  \item $c_{bq}(G)\define \sum_{Q\in\mathcal{Q}}\#\partial Q$ is the
    total number of boundary components of \qh\ vertex groups
  \item $\mathcal{Z}(G)$ is the collection of conjugacy classes of
    centralizers of edge groups of $G$. Warning: \emph{not} the center
    of $G$.
  \item For a given rigid vertex $R$ of the principle cyclic
    \jsj\ decomposition, let $v(R)$ be the valence of $R$.  This is the
    same as the number of conjugacy classes of centralizers of
    incident edge groups \emph{in} $R$.
  \item $c_a(G)\define\sum_{A\in\mathcal{A}}(\rk(A)-1)$
  \item $c_b(G)=c_a(G)+\betti(\Delta)$
  \end{itemize}

  The \term{complexity} of $G$ with respect to
  $\Delta$ is the ordered tuple
  \[
  \jcomp(G,\Delta)=(c_q(G),-c_{bq}(G),\vert\mathcal{Z}\vert,c_b(G),\betti(\Delta),\vert\mathcal{R}\vert,\sum_{R\in\mathcal{R}}v(R)) 
  \]
  The ``,$\Delta$'' is suppressed from the notation if $\Delta$ is the
  principle cyclic \jsj\ of $G$.
\end{definition}

\par Complexities are compared lexicographically.  The complexity
$\jcomp_i$ is the restriction of $\jcomp$ to the first $i$
coordinates.

\emph{Throughout this section $G$ and $H$ are freely indecomposable
  limit groups, $\eta\colon G\into H,$ and $\eta^{\#}\colon
  H^1(H,\zee)\to H^1(G,\zee)$ is injective.}

\par We need to be able to compare the complexity of a principle
decomposition to the complexity of the \jsj.

\begin{lemma}
  \label{lem:identitymap}
  Let $G$ be a freely indecomposable limit group with principle cyclic
  \jsj\ $\Delta_G,$ let $\E$ be a fixed family of subgroups
  of $G,$ and let $\Delta$ be the principle cyclic decomposition of
  $G$ associated to the family of principle cyclic splittings in which
  each $E\in\E$ is elliptic.  Then
  $\jcomp(G,\Delta)\leq\jcomp(G),$ with equality if and only if
  $\Delta$ is the \jsj.
\end{lemma}

\begin{proof}
  We can construct $\Delta$ by cutting \qh\ vertex groups of
  $\Delta_G$ along simple closed curves, folding, and collapsing
  subgraphs.  To handle $c_b,$ observe that any collection of disjoint
  simple closed curves on \qh\ vertex groups of $\Delta$ can be
  completed to a collection which achives at most $c_b(G)$.
  
  The inequalities on $c_q$ and $c_{bq}$ are obvious, and if they are
  equal, then the identity map simply identifies \qh\ vertex
  groups. The remaining inequalities are obvious.
\end{proof}

\par We spread the proof of Theorem~\ref{thm:alignment} across the
next two lemmas.

\begin{lemma}
  \label{surfaceinequality}
  $c_q(G)\geq c_q(H)$.  If equality holds then $c_{bq}(G)\leq c_{bq}(H)$.
\end{lemma}

\begin{proof}
  Let $T$ be the Bass-Serre tree for the restricted \jsj\ of
  $H$.  Since $\eta$ is injective, $G$ inherits a graph of groups
  decomposition $\Delta$ from its action on $T$.  Let $Q$ be a vertex
  group of $\Delta$ conjugate into some element $Q'$ of
  $\mathcal{Q}(H)$.  There are two possibilities: $Q$ either has finite
  or infinite index in $Q'$.  If $Q$ has infinite index and is
  nontrivial then $G$ must be freely decomposable, contrary to
  hypothesis.  Thus $Q$ is either trivial or finite index.

  Let $c$ be a simple closed curve on some element $Q'$ of
  $\mathcal{Q}(H)$ giving a essential one-edged splitting $\Delta_c$
  of $H$.  By Lemma~\ref{actshyperbolically} $G$ acts hyperbolically
  in $\Delta_c,$ hence there is some $Q$ which maps to a finite index
  subgroup of a conjugate of $Q'$.  The graph of groups decomposition
  $\Delta$ of $G$ is obtained by slicing \qh\ vertex groups of $G$
  along simple closed curves, folding, and collapsing subgraphs of the
  resulting decomposition.  This immediately gives $c_q(G)\geq
  c_q(H)$.

  Suppose equality holds.  Let $c_k$ be the simple closed curves
  cutting the \qh\ vertex groups of $\jsj(G),$ and let
  $Q'_1,\dotsc,Q'_m$ be the complementary components which don't map
  to \qh\ vertex groups of $H$.  Since $c_q(G)=c_q(H),$ each component
  $Q'_j$ has Euler characteristic $0$.  Any such complementary
  component cannot be boundary parallel, thus if there are any then
  $c_{bq}(H)>c_{bq}(G)$.  If equality holds then the \qh\ subgroups of
  $G$ and those of $H$ are in one to one correspondence and the
  respective maps are isomorphisms.
\end{proof}

\par An inclusion $G\into H$ as above is \term{\qh--preserving} if it
is a one-to-one correspondence on \qh\ vertex groups and the maps are
isomorphisms. If $H$ has an inessential one-edged splitting $\Delta,$
then $\Delta$ corresponds to an edge connecting a valence one cyclic
vertex group of $\jsj(H)$ to a \qh\ vertex group. If $G\into H$ is
\qh--preserving then it is necessarily bijective on such valence one
vertex groups.

\par It follows immediately from Lemma~\ref{actshyperbolically} that
if $G\into H$ is \qh--preserving then
$\vert\mathcal{Z}(G)\vert\geq\vert\mathcal{Z}(H)\vert$.

\begin{lemma}
  \label{underlyinginequality}
  \label{lem:abelians} 
  \label{centralizerinequality}
  $\jcomp_5(G)\geq\jcomp_5(H)$. If equality holds then there is an
  induced bijection $\mathcal{A}(G)\to\mathcal{A}(H),$ and for each
  $A,$ $A/P(A)\to\eta_{\#}(A)/P(\eta_{\#}(A))$ is virtually onto.
\end{lemma}

\begin{proof}
  We first handle $c_b$.

  Let $\Delta_H$ be the principle cyclic \jsj\ of $H,$ and let
  $\Delta_G$ be the decomposition $G$ inherits from its action on
  $T_{\Delta_H}$.  We may assume that $G\into H$ is \qh--preserving,
  is bijective on conjugacy classes of centralizers of edge groups.
  Let \[H_1(H;\Delta^{(0)}_H\setminus
  \mathcal{A}(H))=H_1(\Delta_H)\oplus\bigoplus_{A\in\mathcal{A}(H)}A/P(A)\]
  Similarly, define $H_1(G;\Delta^{(0)}_G\setminus \mathcal{A}(G))$.
  The composition $G\to H\to H_1(H;\Delta^{(0)}_H\setminus
  \mathcal{A}(H))$ factors through $H_1(G;\Delta^{(0)}_G\setminus
  \mathcal{A}(G))$.  Since $H^1(H,\zee)\into H^1(G,\zee)$ the map
  $H_1(G;\Delta^{(0)}_G\setminus \mathcal{A}(G))\to
  H_1(H;\Delta^{(0)}_H\setminus \mathcal{A}(H)) $ must be virtually
  onto.  But $\rk(H_1(H;\Delta^{(0)}_H\setminus
  \mathcal{A}(H)))=c_b(H)$ and $\rk(H_1(G;\Delta^{(0)}_G\setminus
  \mathcal{A}(G)))\leq c_b(G)$.

  Let $\Delta$ be an essential one-edged splitting of $H$ in which all
  \qh\ subgroups are elliptic.  Let $T$ be the corresponding
  Bass-Serre tree.  By Lemma~\ref{actshyperbolically} $G$ doesn't fix
  a point in $T$ and it inherits an essential splitting $\Delta'$ from
  this action.  Since $\eta$ is bijective of the sets of \qh\
  subgroups, and restricts to isomorphisms between them, every \qh\
  vertex group of $G$ acts elliptically in $T$.  Thus there is an edge
  group $E'$ of $\jsj(G)$ which maps to a conjugate of the edge group
  of $\Delta$. Furthermore, $E'$ is an essential splitting, otherwise
  $G$ acts elliptically in $\Delta$.

  Let $A\in\mathcal{A}(G)$ be a noncyclic abelian vertex group.  If no
  element of $\mathcal{A}(H)$ contains the image of $A,$ then
  $c_b(G)>c_b(H)$. If equality holds there is a well defined map
  $\mathcal{A}(G)\to\mathcal{A}(H)$.

  Let $A$ be an abelian vertex group of $G,$ and $\eta_{\#}(A)$ the
  associated vertex group of $H$. Since $H^1(H)\to H^1(G)$ is
  injective, the map $A/P(A)\oplus H_1(\Gamma_G)\to
  \eta_{\#}(A)/P(\eta_{\#}(A))\oplus H_1(\Gamma_H)$ must be virtually
  onto. This map sends $A/P(A)$ to $\eta_{\#}(A)/P(\eta_{\#}(A))$
  hence $\betti(\Delta_G)\geq\betti(\Delta_H),$ and if
  $\betti(\Delta_G)=\betti(\Delta_H)$ then
  $A/P(A)\to\eta_{\#}(A)/P(\eta_{\#}(A))$ must be virtually onto.
\end{proof}

\begin{theorem}
  \label{thm:alignment}

  $\jcomp_6(G)\geq\jcomp_6(H)$.  If $\jcomp_6(G)=\jcomp_6(H)$ then
  \[\sum_{R\in\mathcal{R}(G)}v(R)\geq\sum_{R\in\mathcal{R}(H)}v(R),\] i.e., $\jcomp(G)\geq\jcomp(H)$.  If $\eta\colon G\into H,$ $\jcomp(G)=\jcomp(H),$ then
  $\eta$ is bijective on vertex and edge groups, maps abelian
  vertex, edge, and peripheral subgroups to finite index subgroups of
  their respective images.  The map from the underlying graph of the
  \jsj\ of $G$ to the underlying graph of the \jsj\ of $H$ is an
  isomorphism.

  The number of values the complexity can take is controlled by
  $\betti$.
\end{theorem}

\begin{proof}[Proof of Theorem~\ref{thm:alignment}]
  Assume $\jcomp_5(G)=\jcomp_5(H)$.  By
  Lemma~\ref{centralizerinequality}, the inclusion is a one-to-one
  correspondence on noncyclic abelian vertex groups.



  Let $\Delta_H$ be the principle cyclic \jsj\ of $H,$ let
  $\pi\colon\Delta_G\to\Delta_H$ be the induced map of underlying
  graphs , and let $R$ be nonabelian non-\qh\ vertex group of
  $\Delta_H$. By Lemma~\ref{actshyperbolically} there is a nonabelian
  vertex group of $\Delta_G$ which maps to $R$. Since $\eta$ is
  bijective on \qh\ subgroups, there is a rigid vertex group $R'$ of
  $G$ which maps to $R$. If $\jcomp_5(G)=\jcomp_5(H)$ then $R'$ is the
  unique such vertex group.


  Again, by Lemma~\ref{actshyperbolically}, since there is only one
  vertex group $R'$ mapping to $R,$ the map
  $\E(R')\to\E(R)$ is onto and $v(R')\geq v(R)$.


  Let $Z$ be an essential cyclic abelian vertex group of $\Delta_H,$
  and let $Z_1,\dotsc,Z_k$ be the vertex groups of $\Delta_G$ mapping
  to $Z$. Since $\eta$ is bijective on nonabelian vertex groups, and
  since all vertex groups adjacent to $Z$ are nonabelian, the induced
  map $\eta_{\#}\colon\sqcup\E(Z_i)\to\E(Z)$ is
  bijective. Arguing as in Lemma~\ref{actshyperbolically}, $k=1$ and
  the map $\E(Z_1)\to\E(Z)$ is bijective. The same
  observation shows that if $A$ is noncyclic abelian, then there is a
  unique $A'$ mapping to $A$ and that the map on the link is onto. The
  map is also injective, again because $\eta$ is a bijective on
  nonabelian vertex groups.

  Thus, if the complexities are equal, then the inclusion must induce
  a homeomorphism of underlying graphs. By construction, the map is
  label preserving, and it automatically respect all incidence and
  conjugacy data from the respective \jsj\ decompositions.  

  This shows that $\jcomp(G,\Delta_G)\geq\jcomp(H),$ and if equality
  holds, then the morphism $\Delta_G\to\Delta_H$ is of the correct
  form. By Lemma~\ref{lem:identitymap}
  $\jcomp(G)\geq\jcomp(G,\Delta_G),$ and if $\jcomp(G)=\jcomp(H),$
  then $\Delta_G$ is just the principle cyclic \jsj\ of $G$.  This
  gives the first half of the theorem.

  The bound on the number of values the complexity can take follows
  from either acylindrical accessibility \cite{sela::acyl} plus the
  bound on the rank of a limit group with complexity $b_0,$
  or~\cite[Lemma~\ref{STABLE-lem:rankheightbound}]{louder::stable},
  which gives a bound on the complexity of the principle cyclic \jsj\
  in terms of the first betti number.  Those arguments bound the
  number of essential vertex groups.  Adjoining roots doesn't increase
  the first betti number, so if $b_1$ and $b_2$ are boundary
  components of a \qh\ vertex group adjacent to inessential vertex
  groups, then a simple closed curve cutting off a pair of pants with
  $b_1,b_2$ as the two other boundary components makes a contribution
  of one to $\betti(G)$; $n$ nonintersecting simple closed curves as
  above make a contribution of $n$ to $\betti(G),$ thus each \qh\
  vertex group is attached to at most $2\betti(G)$ inessential vertex
  groups.  Since $\betti(G)$ controls the number of \qh\ vertex
  groups, there are boundedly many inessential abelian vertex groups.
\end{proof}

In light of Theorem~\ref{thm:alignment}, if $\jcomp(G)=\jcomp(H),$
then we say that $G$ and $H$ are \term{aligned}.  Before representing
injections of limit groups topologically, we devote a section to
proving Theorem~\ref{maintheorem}, assuming the material from
section~\ref{hyptoell}.

\section{Proof of Theorem~\ref{maintheorem}}
\label{sec:mainproof}

\par The bound implicitly computed in the proof of
Theorem~\ref{maintheorem} can be made slightly better if we show that
nonabelian limit groups with first betti number $2$ are free. The next
lemma is not necessary, but we record it here for lack of a better
place to put it. In~\cite{rankthreeclassification}, Fine, et al.,
classify limit groups with rank at most three.  The next lemma shows
that in rank two the rank can be relaxed to first betti number.

\begin{lemma}
  Let $G$ be a limit group with first betti number $2$.  Then
  $G\cong\free_2$ or $\zee^2$.
\end{lemma}

\begin{proof}
  We may assume $G$ is nonabelian and freely indecomposable.  If $G$ is
  abelian it satisfies the theorem trivially, and if freely
  decomposable, the free factors are limit groups with first betti
  number one, and must be infinite cyclic.

  The proof is by induction on the depth of the cyclic analysis
  lattice.  All essential cyclic splittings of $G$ are \hnn\
  extensions, otherwise there is a one-edged cyclic splitting such
  that each vertex group has betti number at least two, and $G$
  therefore has first betti number at least three.  By a simple
  variation of the proof of Theorem~\ref{analysislatticebound} the
  depth of the cyclic analysis lattice of $G$ is finite.  Suppose that
  $G$ has a \qh\ vertex group $Q$.  Then any essential simple closed
  curve on $Q$ must correspond to an \hnn\ extension of $G$:
  $G=G'*_E$.  Since the splitting comes from a \qh\ vertex group, $G'$
  must be freely decomposable, hence is $\free_2$.  If $G$ has no \qh\
  vertex groups it's principle cyclic \jsj\ decomposition must be a
  bouquet of circles.  Let $G=G_0>G_1>G_2>\dotsb>G_n$ be a sequence of
  vertex groups of cyclic \jsj\ decompositions such that $G_i,$
  $i<n-1,$ is freely indecomposable and has a bouquet of circles as
  its principle cyclic \jsj, terminating at the first index $n$ such
  that such that $G_n$ is freely decomposable, hence free.  This chain
  must have finite length since the cyclic analysis lattice is
  finite.  We argue that $G_n$ free implies that $G_{n-1}$ is free.

  Let $f\colon G_{n-1}\to\free$ be a homomorphism such that $f(G_n)$
  has nonabelian image.  Since $G_{n-1}$ is an \hnn\ extension of
  $G_n,$ by Corollary~1.6 of~\cite{louder::scott}, the images of the
  incident edge groups in $G_n$ can be conjugated to a basis for $G_n$
  and $G_{n-1}$ is freely decomposable, contrary to hypothesis.
\end{proof}

\begin{definition}[Extension]
  An \term{extension} of a pure staircase
  $(\G,\HH)$ is a staircase $(\G,\HH')$ such that
  the diagrams in Figure~\ref{fig:extension} commute.  An extension is
  \term{admissible} if one of the following mutually exclusive
  conditions holds.
  \begin{itemize}
  \item $\G$ is freely decomposable, and the freely indecomposable free
    factors of $\HH'(i)$ embed in $\HH(i)$ under $\sigma_i$.
  \item $\G$ is freely indecomposable, has \qh\ subgroups, and the
    vertex groups of the decomposition of $\HH'(i)$ obtained by
    collapsing all edges not adjacent to \qh\ vertex groups embed in
    $\HH(i)$ for all $i$.  (This is just the restricted principle cyclic
    \jsj.)
  \item $\G$ is freely indecomposable, \qh--free, and for all $i,$
    vertex groups of the (restricted) principle cyclic \jsj\ of
    $\HH'(i)$ embed in vertex groups of $\HH(i)$ under $\sigma_i$.
  \end{itemize}
\end{definition}

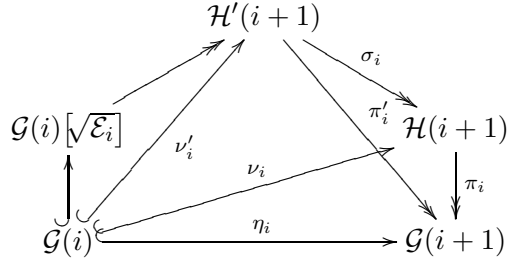
\begin{figure}[h]
  \centerline{%
    \xymatrix{%
      & \HH'(i+1)\ar@{->>}[ddr]^{\pi'_i}\ar@{->>}[dr]^{\sigma_i} \\
      \adjoinroot{\G(i)}{\E_i}{}\ar@{->>}[ur] &  & \HH(i+1)\ar@{->>}[d]^{\pi_i}\\
      \G(i)\ar@{^(->}[rr]^{\eta_i}\ar@{^(->}[uur]_{\nu'_i}\ar@{^(->}[u]\ar@{^(->}[urr]^{\nu_i} & & \G(i+1)
}}
\caption{Extensions of sequences}
\label{fig:extension}
\end{figure}

\par An admissible extension has the property that each $\sigma_i$ is
strict, surjective, and maps elliptic subgroups of a decomposition of
$\HH'(i+1)$ to elliptic subgroups of $\HH(i+1)$.  The relation ``mapps
onto'' partially orders the collection of extensions, and if $\HH''$ is
an extension of $\HH'$ then $\HH''\geq_{\sqrt{}}\HH'$.  For some $i,$ if
$\sigma_i$ is not one-to-one on the sets of vertex groups or edge
groups then the inequality is strict.  The envelope of a rigid vertex
group of the principle cyclic \jsj\ is just the vertex group, hence if
$\sigma_i$ is one-to-one on the sets of vertex groups and edge groups
then it is an isomorphism.  If $(\G,\HH'')\onto(\G,\HH')\onto(\G,\HH)$ is
a pair of admissible extensions then $(\G,\HH'')\onto(\G,\HH)$ is an
admissible extension.

\par We work with staircases which are maximal with respect to
$\geq_{\sqrt{}},$ rather than arbitrary staircases.  To do this we
have to pay a penalty, but not too large of one.

\begin{lemma}
  \label{existresolutions}
  For all $K$ there exists $C=C(K,\comp(\text{ }))$ such that if
  $(\G,\HH,\E)$ is a staircase and
  $\ninj(\comp((\G,\HH,\E)))=C(K,\comp((\G,\HH,\E))),$
  then there is a $\geq_{\sqrt{}}$--maximal extension of a contraction
  $(\G',\HH',\E')$ of $(\G,\HH,\E)$ with
  $\ninj((\G',\HH',\E'))\geq K$ and
  $\comp((\G',\HH',\E'))\leq\comp((\G,\HH,\E))$.
\end{lemma}

\par The constants in this lemma do not depend on
$\Vert\E\Vert,$ and its proof is formally identical to the
proof of \cite[Theorem~\ref{STABLE-thr:nostrict}]{louder::stable}.  To
adapt the proof, we need to show that the strict resolutions arising
in an extension have bounded length.  This follows
from~\cite[Lemma~\ref{STABLE-lem:rankheightbound}]{louder::stable},
bounding the rank of $\HH(i)$ from above by a function of
$\comp((\G,\HH)),$ but a proof more in the spirit of this paper goes as
follows: If $\HH^{(n)}(i)\onto\dotsb\onto\HH(i)$ is a strict resolution
appearing in a sequence of extensions, then
$\jcomp(G)\geq\jcomp(\HH^{(m)}(i))$ (See Lemma~\ref{thm:alignment} and
Definition~\ref{def:compjsj}.), moreover, if
$\HH^{(m+1)}(i)\onto\HH^{(m)}(i)$ is not injective on sets of vertex or
edge spaces, or collapses subsurface groups of \qh\ vertex groups, the
complexity must decrease.  By Theorem~\ref{thm:alignment} the number
of values the complexity takes is controlled by $\betti,$ and the
resolutions have length controlled by $\comp((\G,\HH,\E))$.

\mnote{jcomp decreases along the resolution.}

\mnote{Suppose Theorem~\ref{maintheorem} is false.  Then by the lemma,
  for some smallest complexity $(b_0,d_0,e_0)$ there exist maximal
  cyclic staircases with arbitrarily high $\ninj$.}\mnote{Where does
  this find a minimal complexity crap go?}

\par Each pure kind of staircase is handled in
turn over the next three subsections.  In all cases the strategy is the
same: either there is compatibility between (collapses of) \rjsj\
decompositions/Grushko factorizations, the complexity decreases,
or proper extensions exist.

\subsection{Freely decomposable}

\par This is the most singular case in that the arguments work for
nearly all finitely generated groups, not just limit groups.

\par The complexity for freely indecomposable groups is used to show
that base sequences of freely indecomposable staircase can be divided
into segments such that the base groups of a segment have the same
\jsj\ decompositions, in the sense of Theorem~\ref{thm:alignment}.
There is a similar complexity for freely decomposable groups which
accomplishes the same thing but with regard to Grushko decompositions.
The following theorem from~\cite{louder::scott} shows how the
complexity for freely decomposable groups is useful.

\begin{definition}[Scott complexity]
  Let $G$ be a finitely generated group with Grushko decomposition
  $G=G_1*\dotsb*G_p*\free_q$.  The \term{Scott
    complexity} of $G$ is the lexicographically ordered pair
  $\scott(G)\define(q-1,p)$.
\end{definition}

\par The number of Scott complexities of limit groups with $\betti=b$
is bounded by $b^3$.

\begin{theorem}[Scott complexity and adjoining roots to groups]
  \label{roots::freelydecomposable}
  Suppose that $\phi\colon G\into H$ and $H$ is a quotient of
  $G'=\adjoinroot{G}{\gamma_i}{k_i},$ $\boldsymbol{\gamma}_i$ a
  collection of distinct conjugacy classes of indivisible elements of
  $G$ such that $\boldsymbol{\gamma}_i\neq\boldsymbol{\gamma}_j^{-1}$
  for all $i,j$ and $\gamma_i\in\boldsymbol{\gamma}_i$.  Then
  $\scott(G)\geq\scott(H)$.  If equality holds and $H$ has no $\zee_2$
  free factors, then there are presentations of $G$ and $H$ as
  \[
    G\cong\grushko{G}{p}{q}^G,\quad H\cong\grushko{H}{p}{q}^H
    \]
  a partition of $\set{\boldsymbol{\gamma}_i}$ into subsets
  $\boldsymbol{\gamma}_{j,i},$ $j=0,\dotsc,p,$ $i=1,\dotsc,i_p,$
  representatives $\gamma_{j,i}\in G_j\cap\boldsymbol{\gamma}_{j,i},$
  $i\geq 1,$ $\gamma_{0,i}\in\free_q^G\cap\boldsymbol{\gamma}_{0,i},$
  such that with respect to the presentations of $G$ and $H$:
  \begin{itemize}
    \item $\phi(G_i)<H_i$
    \item $\adjoinroot{G_j}{\gamma_{j,i}}{k_{j,i}}\onto H_j$
    \item $\phi(\free_q^G)<\free_q^H$
    \item
    $\free^G_q=\group{\gamma_{0,1}}*\dotsb*\group{\gamma_{0,i_0}}*F$
    \item
      $\free^H_q=\group{\sqrt{\gamma_{0,1}}}*\dotsb*\group{\sqrt{\gamma_{0,i_0}}}*F$
    \item
    $G'\cong\adjoinroot{G_1}{\gamma_{1,i}}{}*\dotsb*\adjoinroot{G_p}{\gamma_{p,i}}{}*\group{\sqrt{\gamma_{0,1}}}*\dotsb*\group{\sqrt{\gamma_{0,i_0}}}*F$
  \end{itemize}
  All homomorphisms are those suggested by the presentations, and the
  maps on $F$ are the identity.
\end{theorem}

\par This is~\cite[Theorem~1.2]{louder::scott}.

\begin{remark}
  Theorem~\ref{roots::freelydecomposable} is stated in terms of
  adjoining roots to cyclic subgroups of a group, whereas
  Definition~\ref{def:adjunctionofroots} refers to collections of
  abelian subgroups.  This difference is immaterial to the discussion
  here since adjoining roots to a noncyclic abelian group can be
  accomplished by adjoining roots to a suitable collection of cyclic
  subgroups.  By passing from a noncyclic abelian subgroup to cyclic
  subgroups, the measure $\Vert\text{ }\Vert$ is unchanged.
\end{remark}

\begin{definition}[Free products]
  Let $(\G_i,\HH_i)$ be a collection of staircases on the same index
  set $I$.  Then the graded free product $((*_i\G_i),(*_i\HH_i)),$ with
  the obvious maps, is also a sequence of adjunctions of roots.
\end{definition}

\begin{lemma}
  \label{maintheorem::freelydecomposable}
  Suppose Theorem~\ref{maintheorem} holds for all staircases with
  complexity less than $(b_0,d_0,e_0)$.  Then Theorem~\ref{maintheorem}
  holds for pure freely decomposable staircases of complexity
  $(b_0,d_0,e_0)$.
\end{lemma}

\begin{proof}
  Let $(\G,\HH,\E)$ be a staircase with complexity
  $(b_0,d_0,e_0)$.  Since limit groups are torsion free, no $\G(i)$ has
  a $\zee_2$ free factor, and by
  Theorem~\ref{roots::freelydecomposable} for all but
  $\betti(\G(1))^3$ indices $i_j,$ the subsequences
  $\G(i_j)\into\G(i_j+1)\into\dotsb\into\G(i_{j+1}-1)$ can be
  decomposed into free products of freely indecomposable groups
  staircases.  Moreover, elements of $\E_i$ are either part of
  a basis of a free free factor of $\G(i)$ or are conjugate into a
  freely indecomposable free factor of $\G(i)$.  Write $\G(i)$ as the
  free product \[\G(i)_1*\dotsb*\G(i)_p*F_i\] given by the lemma,
  where $F_i$ is a free group of rank $q$ and $\scott(\G(i))=(q-1,p)$
  for all $i$.  Let $\E^j_i$ be the subset of
  $\E_i$ consisting of elements conjugate into $\G(i)_j,$ and
  rearrange indices so that $\G(i)_j$ maps to $\G(i+1)_j$ for all
  $j$.  Let $\E^0_{i}$ be the elements of $\E$ which
  are conjugate into $F_i$.  By
  Theorem~\ref{roots::freelydecomposable} there are
  decompositions 
  \[
    \adjoinroot{\G(i)}{\E_i}{}\cong\left(\adjoinroot{\G(i)_1}{\E^1_{i}}{}*\dotsb*\adjoinroot{\G(i)_p}{\E^p_{i}}{}\right)*\adjoinroot{F_i}{\E^0_{i}}{}
    \]
  where the last factor is free.  Let
  $\HH(i+1)_j\define\img_{\HH(i+1)}(\adjoinroot{\G(i)_j}{\E^j_{i}}{})$
  The sequence $\HH'$ defined
  by 
  \[
    \HH'(i+1)\define\left(*_j\HH(i+1)_j\right)*\adjoinroot{F_i}{\E^0_{i}}{}
    \]
  is an extension of $\HH$.  Then $(\G,\HH',\E)$ splits as a
  free product, the freely indecomposable free factors of which are
  $(\G_j,\HH'_j,\E^j)$.  These free factors have strictly lower
  $\betti$ than $\G,$ depth at most $d_0=\depthpc(\HH),$ hence have
  $\ninj((\G_j,\HH'_j,\E^j))\leq
  \ninj(\comp(b_0-1,d_0,e_0))=:B$.  If $\Vert\G\Vert>B\cdot
  \betti(\G)\geq B\cdot p,$ then, for some index $l,$ the map
  $\HH'(l)\onto\G(l)$ is visibly an isomorphism.  Since this map factors
  through $\HH(l),$ $\HH(l)\onto\G(l)$ is an isomorphism as well.
\end{proof}

\par We finish this subsection by proving the base case of the
induction. Let $(\G,\HH,\E)$ be a maximal staircase of
complexity $(b,2,e)$. By the proof of
Lemma~\ref{roots::freelydecomposable}, the staircase splits as a free
product of freely indecomposable staircases
$(\G_i,\HH_i,\E^i),$ and such that each $\HH_i(j)$ is
elementary. If $\G_i$ is abelian, then clearly $\HH_i(j)\onto\G_i(j)$
is an isomorphism, and if nonabelian, $\HH_i(j)$ is the fundamental
group of a closed surface. Since $\G_i$ is freely indecomposable, it
is also the fundamental group of a closed surface. Divide $\G_i$ into
segments such that the Euler characteristic is constant on each
segment. Then $\G_i(j)\to\G_i(j+1)$ is an isomorphism on each segment
and $\HH_i(j)$ is a trivial extension of $\G_i(j-1)$ for all $j$ on
each segment, thus $\HH_i(j)\onto\G_i(j)$ is an isomorphism.

\subsection{Freely indecomposable, \qh}

\par Lemma~\ref{onetooneonedges} allows us to handle injections
$G\into H,$ $\jcomp(G)=\jcomp(H),$ and such that $G$ has a
\qh\ subgroup.

\begin{lemma}
  \label{alignedandqh}
  Suppose Theorem~\ref{maintheorem} holds for all staircases with
  strictly lower complexity than $(b_0,d_0,e_0)$.  Then
  Theorem~\ref{maintheorem} holds for staircases with \qh\ subgroups
  and complexity $(b_0,d_0,e_0)$.
\end{lemma}

\par The strategy is to find an extension $(\G,\HH',\E)$ of
$(\G,\HH,\E)$ such that the $\qh$ subgroups of $\HH'$ are the
``same'' as those from $\G$.  See Figure~\ref{qhexample}.  The group
$\HH(i)$ may be a total mess, but luckily it is a homomorphic image of
a limit group which shares its restricted \jsj\ with $\G(i)$ and
$\G(i+1)$.

\par To do this an auxiliary lemma which follows immediately from
Lemma~\ref{onetooneonedges} is needed.


\begin{lemma}
  \label{dontpassthroughqh}
  Let $G'$ be obtained from $G$ by adjoining roots to a collection of
  abelian subgroups $\E$.  If $\jcomp(G)=\jcomp(G')$ then
  every element $E\in\E$ such that $[E:F(E)]>1$ is conjugate
  into a non-\qh\ vertex group of $\rjsj(G)$.
\end{lemma}

\par We use the immersion representing $G\into G'$ constructed in
subsection~\ref{subsec:graphsandimmersions}.

\begin{proof}
  Fix $E$ as in the statement of the lemma.  We are done if we show
  that $E$ is elliptic in every one edged splitting of $G$ obtained by
  cutting a \qh\ subgroup along an essential simple closed curve which
  doesn't cut off a \mobius\ band.\footnote{We could have instead
    redefined an essential curve as one which gives a principle cyclic
    splitting and isn't boundary parallel.}  Start with an immersion
  representing the \rjsj\ decompositions of $G$ and $G',$ and let
  $\Sigma_Q$ be the surface which contains $c$.  There is a unique
  element $\eta_{\#}(Q)$ containing the image of $Q,$ and the map
  $Q\to\eta_{\#}(Q)$ is surjective.  Since $Q\to\eta_{\#}(Q)$ is
  represented by a homeomorphism $\Sigma_Q\to\Sigma_{\eta_{\#}(Q)}$
  there is a simple closed curve $\eta_{\#}(c)$ contained in
  $\Sigma_{\eta_{\#}(Q)}$ and a closed annular closed neighborhood $A$
  of $c$ mapping homeomorphically to a neighborhood of $\eta_{\#}(c)$.
  Use these neighborhoods to construct new graphs of spaces $Y_G$ and
  $Y_{G'}$ representing $G$ and $G'$ by regarding the annulus as a new
  edge space and collapsing all but the newly introduced edges.  By
  construction, the map $Y_G\into Y_{G'}$ is an immersion.  By
  Lemma~\ref{onetooneonedges}, if some element of $\E$ crosses $c,$
  then $c$ maps to a power of $\eta_{\#}(c)$.
\end{proof}

\begin{figure}[h]
  \psfrag{Gi}{$\G(i)$}
  \psfrag{Gip}{$\G(i+1)$}
  \psfrag{Hi}{$\HH(i)$}
  \psfrag{Hip}{$\HH'(i)$}
  \psfrag{Vi}{$V^i_j$}
  \psfrag{Vip}{$V^{i+1}_j$}
  \psfrag{Wi}{$W^i_j$}
  \psfrag{pi}{$\pi_i$}
  \psfrag{pip}{$\pi'_i$}
  \psfrag{sigma}{$\sigma_i$}
  \psfrag{nup}{$\nu'_i$}
  \psfrag{nu}{$\nu_i$}
  \psfrag{eta}{$\eta_i$}
  \centerline{%
    \includegraphics{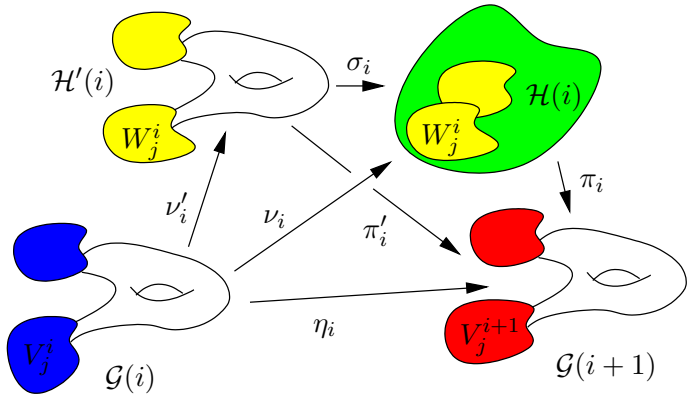}
}
\caption{Illustration of Lemma~\ref{alignedandqh}}
\label{qhexample}
\end{figure}

\begin{proof}[Proof of Lemma~\ref{alignedandqh}]
  Suppose $(\G,\HH)$ is a staircase such that $\scott(\G(i))$ is the
  constant sequence and $c_q(\G(1))\neq 0$.  Let $\Delta_i$ be the
  $\rjsj$ of $\G(i)$.  Every edge of $\Delta_i$ is infinite cyclic and
  connects a vertex group to a boundary component of a \qh\ vertex
  group.  Since the inclusions $\G(i)\into\G(i+1)$ respected graphs of
  spaces, by the first part of Lemma~\ref{dontpassthroughqh}, every
  element of $\E_i$ is conjugate into some non-\qh\ vertex
  group of $\Delta_i$.  Let $V^i_1,\dotsc,V^i_n$ be the non-\qh\ vertex
  groups of $\Delta_i$.  We regard $\G(i)$ as a graph of groups
  $\Gamma(V^i_j,Q_k,E_l),$ where $\G(i)\into\G(i+1)$ is compatible
  with the decomposition $\Gamma$ in the sense that $V^i_j$ maps to a
  conjugate of $V^{i+1}_j,$ the map respects edge group incidences,
  and the inclusion is the identity on the \qh\ vertex groups $Q_k$.

  Let $\E^j_i$ be the elements of $\E_i$ conjugate
  into $V^i_j,$ and arrange that each $E\in\E^j_i$ is
  contained in $V^i_j$ by conjugating if necessary.  Let $W^{i+1}_j$ be
  the image of $\adjoinroot{V^i_j}{\E^j_i}{}$ in $\HH(i+1)$
  and let $\HH'(i+1)=\Gamma(W^{i+1}_j,Q_k,E_l)$.  Then
  $(\G,\HH',\E)$ is an extension of $(\G,\HH,\E)$: The
  map implicit map $\sigma_i\colon\HH'(i)\to\HH(i)$ is clearly strict,
  therefore the sequence $\mathcal{H}'$ consists of limit groups.  By
  definition, $V^{i+1}_j$ is obtained from $V^i_j$ by adjoining
  roots.  Let $\mathcal{V}_j$ be the sequence $\mathcal{V}_j(i)=V^i_j$
  and let $\mathcal{W}_j(i)=W^i_j$.

  The staircases $(\mathcal{V}_j,\mathcal{W}_j,\E^j)$ all
  have lower first betti number than $(\G,\HH,\E)$.  Let
  $B(b_0)$ be the maximal number of vertex groups of a limit group
  with first betti number
  $b_0$~\cite[Lemma~\ref{STABLE-lem:rankheightbound}]{louder::stable}.
  If $\Vert\G\Vert>\ninj((b_0-1,d_0,e_0))\cdot B(b_0)$ then for at
  least one index $l$ all $\mathcal{W}_j(l)\onto\mathcal{V}_j(l+1)$
  are injective.  Thus $\pi'_l$ is $\Mod(\HH'(l),\rjsj)$ strict.  Since
  all modular automorphisms of $\HH'(l)$ are either inner, Dehn twists
  in boundary components of \qh\ vertex groups, or induced by boundary
  respecting homeomorphisms of surfaces representing \qh\ vertex
  groups, by construction, every element of $\Mod(\HH'(l),\rjsj)$
  pushes forward to a modular automorphism of $\G(l)$.  An easy
  exercise shows that $\HH'(l)\onto\G(l)$ is an isomorphism.  Since
  $\pi'_l=\pi_l\circ\sigma_l,$ $\pi_l$ is an isomorphism.
\end{proof}

\subsection{Freely indecomposable, no \qh}

\par The neighborhood of a vertex group $V$ of a graph of groups
decomposition is the subgroup generated by $V$ and conjugates of
adjacent vertex groups which intersect $V$ nontrivially, and is
denoted $\nbhd(V)$.  Let $(G,H,G')$ be a flight and suppose $G$ is
freely indecomposable, has no \qh\ vertex groups, and that
$\jcomp(G)=\jcomp(G')$.  Let $\eta\colon G\to G'$ be the inclusion
map.  An abelian vertex group $A$ of $G$ is \term{$H$--elliptic} if
$H$ doesn't have a principle cyclic splitting over a subgroup of
$\cent_H(\nu(A))$.

\par Let $\mathcal{A}_H$ be the collection of abelian vertex groups of
$G$ which are $H$--elliptic.  Suppose that $H$ is obtained from $G$ by
adjoining roots to the collection $\E$.  Let
$\E^{ell}_H$ be the sub-collection of $\E$
consisting of elements of $E$ which are hyperbolic in the principle
cyclic \jsj\ of $G$ but which have elliptic image in the principle
cyclic \jsj\ of $H$.  Let $\jsj_H(G)$ be the \jsj\ decomposition of $G$
with respect to the collection of principle cyclic splittings in which
all $\nbhd(A),$ $A\in \mathcal{A}_H,$ and $E\in\E^{ell}_H$
are elliptic:
\[
\jsj_H(G)\define \jsj(G;\set{\nbhd(A),E}:{A\in \mathcal{A}_H,E\in\E^{ell}_H}) 
\]
Let $\jsj^*_H(G')$ be the \jsj\ decomposition of $G'$ associated to
the collection of all principle cyclic splittings of $G'$ in which all
$\eta_{\#}(A_H),$ $A\in \mathcal{A}_H,$ and $\eta_{\#}(E),$
$E\in\E^{ell}_H,$ are elliptic.  That is
\[
\jsj^*_H(G')\define \jsj(G';\set{\nbhd(\eta_{\#}(A)),\eta_{\#}(E)}:{A\in \mathcal{A}_H, E\in\E^{ell}_H} )
\]

\par The main lemma is that the decompositions of $G$ and $G'$ induced
by $H$ are intimately related to the principle cyclic \jsj\ of $H$ as
long as the flight admits no proper extensions.  Let $V$ be a vertex
group of $\jsj_H(G)$.  There is a vertex group $\eta_{\#}(V)$ of
$\jsj^*_H(G')$ which contains the image of $V$.  Let $\E_V$
be the collection of elements of $\E$ which are conjugate
into $V,$ along with the collection of incident edge groups.  Likewise
for $\eta_{\#}(V),$ let $\E(\eta_{\#}(V))$ be the set of
centralizers of images of elements of $\E_V$.

\begin{lemma}
  \label{mainlemma}
  Let $(G,H,G')$ be a flight without any proper
  extensions, and suppose $G$ is freely indecomposable, has no
  \qh\ vertex groups, and that $\jcomp(G)=\jcomp(G')$.  Let $\eta\colon
  G\to G'$ be the inclusion map.  Then the following hold:
  \begin{itemize}
    \item For each vertex group $W$ of the principle cyclic \jsj\ of
      $H$ there are unique vertex groups $V$ and $\eta_{\#}(V)$ of
      $\jsj_H(G)$ and $\jsj^*_H(G'),$ respectively, such that
      $\nu(V)<W,$ $\pi(W)<\eta_{\#}(V)$.
    \item $W$ is obtained from $V$ by adjoining roots to
      $\E_V$ and
      $\Vert\E_V\Vert\leq\Vert\E\Vert+2\betti(G).$
    \item $\eta_{\#}(V)$ is obtained from $\pi(W)$ by adjoining roots
      to the images of $\E(V)$ (the edge groups incident to $V$)
    \item If $\pi$ is injective on vertex groups then it is an
      isomorphism.
  \end{itemize}
\end{lemma}

\par The proof of Lemma~\ref{mainlemma} is contained in
section~\ref{hyptoell}, where graphs of spaces $X_G,$ $X_H,$
representing $\jsj_H(G)$ and $\jsj(H),$ respectively, and an immersion
$X_G\to X_H$ representing $G\into H,$ such that if the immersion is
not one-to-one on edge spaces, then there must be a nontrivial
extension, are constructed.  The remainder of the lemma is largely
formal, and relies on a simplification of the construction of strict
homomorphisms from~\cite{louder::stable}.

\subsection{Finishing the argument}

\label{finishargument}

\par In this section we prove Theorem~\ref{maintheorem}, postponing
the proofs of lemmas used in the previous section until
section~\ref{hyptoell}.  Let $(\G,\HH)$ be a staircase with complexity
$(b_0,d_0,e_0),$ such that no contraction has any proper extensions,
and suppose that Theorem~\ref{maintheorem} holds for staircasess with
complexity less than $(b_0,d_0,e_0)$.  By Theorem~\ref{thm:alignment}
there is some constant $B(b_0)$ such that $(\G,\HH)$ can be divided
into $B(b_0)$ staircases of constant Scott complexity: (To maintain
uniformity of the exposition, some sequences are allowed to be empty.)

\[(\G,\HH)\mapsto\set{(\G_i,\HH_i)}_{i=1,\dotsc,B(b_0,d_0)}\]
\[\G_i(1)=\G(j_i),\dotsc\quad\HH_i(2)=\HH(j_i+1),\dotsc\]

\par Only the last of these can consist of freely indecomposable
groups.  Each staircase $(\G_i,\HH_i),$ $i<B(b_0),$ by
Theorem~\ref{maintheorem::freelydecomposable}, has $\ninj$ bounded
above by $b_0\cdot \ninj(b_0-1,d_0,e_0),$ since there are at most
$b_0$ freely indecomposable free factors.  Thus we may confine our
analysis to freely indecomposable staircases.  By
Theorem~\ref{thm:alignment}, we may divide the staircase
$(\G,\HH,\E)$ into boundedly many segments, the number
depending only on the complexity of $\betti(\G),$ exhausting the
tower, such that $\jcomp$ is constant on each segment.  By
Lemma~\ref{existresolutions} we may assume that each segment is
maximal with respect to $\leq_{\sqrt{}}$.

\par Like the case when each $\G(i)$ is freely decomposable, if
$\G(i)$ has a \qh\ vertex group, by Lemma~\ref{alignedandqh} such
staircases have bounded $\ninj$.

\par The only possibility left is that the contractions of $(\G,\HH)$
are \qh--free.  Let $I$ be the index set for $\G,$ and color the
triple $i<j<k$ red if $\jsj^*_{\HH(j)}(\G(j))\cong\jsj_{\HH(k)}(\G(j)),$
and blue otherwise.  Then by Ramsey's theorem for hypergraphs, for all
$K$ there exists an $L$ such that if $\Vert\G\Vert>L$ then there is a
subset $I'\subset I$ of size at least $K$ such that all triples whose
elements are in $I'$ have the same color.

\begin{lemma}
  \label{lem:boundblue}
  There is an upper bound to the size of blue subsets which depends
  only on $\betti(\G)$ and $\Vert\E\Vert$.
\end{lemma}

\begin{proof}
  By Lemma~\ref{lem:numberofdecompositions}, there are at most
  $2^{\Vert\E\Vert}$ equivalence classes of principle cyclic
  decomposition of $G$ in which some element of $\E$ is
  elliptic.  (There may be none.) Suppose $\vert
  I'\vert>2^{\Vert\E\Vert},$ and consider the collection of
  principle cyclic decompositions $\set{\jsj_{\HH(l)}(\G(i))}$.  Thus, for
  some $i<j<k,$ $\jsj_{\HH(j)}(\G(i))$ and $\jsj_{\HH(k)}(\G(i))$ have
  the same elliptic subgroups.  Then
  $\jsj_{\HH(k)}(\G(j))\cong\jsj^*_{\HH(j)}(\G(j))$ since a \jsj\
  decomposition is determined up to equivalence solely by its elliptic
  subgroups.
\end{proof}

\par We are now on the home stretch.  Suppose again that
$(b_0,d_0,e_0)$ is the lowest complexity for which
Theorem~\ref{maintheorem} doesn't hold.  By Lemma~\ref{lem:boundblue}
and the prior discussion, there must be staircases
$(\G,\HH,\E)$ of arbitrary $\ninj,$ which have complexity
$(b_0,d_0,e_0),$ are maximal, pure, and have no \qh\ vertex groups.

\par Let $(\G,\HH)$ be such a staircase.  Let $\V'_j(i)$ be the
nonabelian vertex groups of $\G(i),$ indexed such that $\V'_j(i)$ maps
to $\V'_j(k)$ for all $k>i$.  Let $\W_j(i)$ be the corresponding rigid
vertex group of $\HH(i)$.  By the second bullet of
Lemma~\ref{mainlemma}, $\W_j(i+1)$ is obtained from $\V'_j(i)$ by
adjoining roots to $\E^{j,\prime}_i,$ the set of elements of
$\E_i$ which are conjugate into $\V'_j(i),$ along with the
incident edge groups.

\par Let $\V_j(i)<\V'_j(i)$ be the image of $\W_j(i)$ in $\G(i)$.  By
the third bullet of Lemma~\ref{mainlemma}, $\V'_j(i)$ is obtained from
$\V_j(i)$ by adjoining roots to the images of the edge groups incident
to $\W_j(i)$.  Let
\[
\E^{j}_i\define\\
\set{E\cap \V_j(i)}:{E\in \E^{j,\prime}_i}\cup\\
\set{E\cap\V_j(i)}:{E\in\E(\V_j(i))}
\]
The incident edge groups are cyclic, and we can build $\W_j(i)$ by
simply adjoining roots to $\E^{j}_i$ in $\V_j(i)$.  Then
$\E^{j}_i$ is larger than $\E$ by at most the
number of edge groups incident to $\V'_j(i),$ which is at most
$2\betti(\mathcal{G})$.  That is,
\[
  \Vert\E^{j}_i\Vert\leq\Vert\E\Vert+2\betti(\G)
  \leq
  \Vert\E\Vert+2\betti(\G)(\depthpc(\HH)-\depthpc(\W_j))
\]


\par Given a sufficiently long \qh--free staircase
$(\G,\HH,\E),$ we passed to a maximal extension (which we will
also call $(\G,\HH,\E)$) of a substaircase of prescribed
length, such that the sequences of vertex groups
$(\V_j,\W_j,\E^{j})$ of the extension were cyclic staircases.
The vertex groups of the extension are subgroups of the vertex groups
of $\HH,$ hence the depth of $\W_j(i)$ is strictly less than the depth
of $\HH$.  Moreover, the first betti number of $\W_j(i)$ is at most
$\betti(\HH)$ and by Lemma~\ref{mainlemma},
$\comp((\G,\HH,\E))>\comp((\V_j,\W_j,\E^j))$.  There
is an upper bound $B(b_0)$ to the number of vertex groups of the
principle cyclic \jsj\ of a limit group with first betti number $b_0$.
If $\Vert\G\Vert>B(b_0)\cdot \ninj(b_0,d_0-1,e_0+2b_0)$ there is some
index $l$ such that $\HH(l)\onto\G(l)$ is injective on all vertex
groups.  By the last bullet of Lemma~\ref{mainlemma},
$\HH(l)\onto\G(l)$ is an isomorphism.


\section{Hyperbolic to elliptic}
\label{hyptoell}

\subsection{Graphs of spaces and immersions}
\label{subsec:graphsandimmersions}

\par In this section we are given a fixed flight $(G,G',H)$ of limit
groups.  By a \term{graph of spaces representing a principle cyclic
  decomposition} of a limit group $G$ we mean a graph of spaces of the
following form:

\begin{itemize}
  \item For each rigid vertex group $R$ a space $X_R$.  Let
    $\E(R)$ be the edge groups incident to $R,$ and for each
    $E\in\E(R)$ let $\sqrt{E}$ be the maximal cyclic subgroup
    of $R$ containing the image of $E$.  For each $E\in\E$
    there is an embedded copy $S_E$ of $S^1$ in $X_R$ representing the
    conjugacy class of $\sqrt{E}$.
  \item For each edge $E,$ a copy $T_E$ of $S^1,$ with basepoint $b_E$
    and an edge space $T_E\times\mathrm{I}$.  On occasion we confuse
    $T_E$ with $T_E\times\frac{1}{2},$ and sometimes refer to $T_E$ as
    the edge space.  The interval $b_E\times I$ is denoted $t_E,$ and
    we choose an arbitrary orientation for $t_E$.  The end of the edge
    space associated to $E$ is attached via the covering map
    $T_E\immerses S_E$ representing $E\into\sqrt{E}$.
  \item For each abelian vertex group a torus $T_A$.  If $A$ is
    infinite cyclic then $T_A$ has a basepoint $b_A$ and the incident
    edge maps are simply covering maps which send $b_E$ to
    $b_A$.  These covering maps are isomorphisms unless the edge is
    adjacent to a \qh\ vertex group, in which case they may be
    proper.  For each edge space edge $E$ adjacent to $A,$ an edge
    space $T_E\times I$ and an embedded copy of $T_E$ in $T_A$.  This
    assumes edge groups not adjacent to \qh\ vertex groups are
    primitive.  Though there may be \qh\ vertex groups, the cases which
    this definition is designed to handle do not, and we let this
    inconsistency slide.  Unlike the rigid case, the embedded $T_E$
    need not be disjoint, though if they meet, they coincide.  We
    require that any two embedded $T_E,$ differ by an element of
    $T_A,$ treated now as a group.
  \item For each \qh\ vertex group $Q$ a surface with boundary $\Sigma_Q$.
  \item If an edge group $E$ is incident to a \qh\ vertex group $Q$
    then $T_E$ is identified with a boundary component of $\Sigma_Q$.
  \item The resulting graph of spaces has the fundamental group of $G$.
\end{itemize}

\par Let $\eta\colon G\into H$ be an inclusion of limit groups, and
let $\Pi_G$ and $\Pi_H$ be principle cyclic decompositions of $G$ and
$H,$ respectively, such that $\eta$ maps vertex groups to vertex
groups, edge groups to edge groups, and respects edge data, i.e., if
$E\into V,$ $\eta_{\#}(E)\into\eta_{\#}(V),$ then the obvious square
commutes.  If this is the case then $\eta$ \term{respects} $\Pi_G$ and
$\Pi_H$.  Let $X_H$ be a graph of spaces representing $\Pi_H$.  Then
there is a principle cyclic decomposition $\Pi_G$ of $G,$ a space
$X_G$ representing $\Pi_G,$ and an \term{immersion} $\psi\colon X_G\to
X_H,$ inducing $\eta,$ of the following form:

\begin{itemize}
\item For each abelian vertex group $A$ there is a finite sheeted
  covering map $\psi\vert_{T_A}\colon T_A\immerses
  T_{\eta_{\#}(A)}$.  The inclusions of incident edge spaces are
  respected by $\eta$:
  \[\psi\vert_{\img(T_E)}=(T_{\eta_{\#}(E)}\into T_{\eta_{\#}(A)})\circ\psi\vert_{T_E}\]
\item For each $E$ there is a finite sheeted product-respecting
  covering map $T_E\times\mathrm{I}\immerses
  T_{\eta_{\#}(E)}\times\mathrm{I}$ which maps $t_E$ to
  $t_{\eta_{\#}(E)}$.  If $E$ is adjacent to a \qh\ vertex group then
  the degree of the covering map is one.
\item For each $R$ there is a map $X_R\to X_{\eta_{\#}(R)}$
  such that for each edge group $E$ incident to $R$ the following
  diagram commutes:

  \centerline{
    \xymatrix{
      T_E\times\set{0}\ar[r]\ar[d] & X_R\ar[d] \\
      T_{\eta_{\#}(E)}\times\set{0}\ar[r] & X_{\eta_{\#}(R)}
    }}
Likewise for $\times\set{1}$.

\item For each $\Sigma_Q$ a homeomorphism
  $\Sigma_Q\to\Sigma_{\eta_{\#}(Q)}$.  The maps $X_R\to
  X_{\eta_{\#}(R)}$ (similarly for $T_A$'s) respect attaching maps
  of boundary components of surfaces.

\item If $E_1$ and $E_2$ are incident to $A$ and $T_{E_1}$ and
  $T_{E_2}$ have the same image in $T_A,$ then
  $\eta_{\#}(E_1)\neq\eta_{\#}(E_2)$.
\end{itemize}

\par The existence of immersions as above is an easy variation on
Stallings's folding. One way to construct immersions of graphs
representing subgroups is to pass to the cover of a graph representing
a subgroup and trimming trees. There is an analogous construction in
this context.

\subsection{Roots, immersions, and resolving}

\par We need to be able to represent conjugacy classes of elements of
limit groups as nice paths in graphs of spaces.

\begin{definition}[Edge path]
  Let $X_G$ be a graph of spaces representing a principle cyclic
  decomposition of $G$.  The \term{zero skeleton} of $X_G,$ denoted
  $X^0_G,$ is the union of vertex spaces.

  An \term{edge path} in a graph of spaces $X_G$ is a map
  $p\colon\left[0,1\right]\to X_G$ such that $p^{-1}(X^0_G)$ contains
  $\set{0,1}$ and is a disjoint collection of closed subintervals.  Let
    $\left[a,b\right]$ be the closure of a complementary component of
    $p^{-1}(X^0_G)$.  Then $p$ maps $\left[a,b\right]$ homeomorphically
    to some $t_E$.
 
    Let $X_R$ be a vertex space.  Set $\partial X_R$ be the union of
    copies of edge spaces contained in $X_R$.  An edge path $p$ is
    reduced if every restriction
    $p\vert_{\left[a,b\right]}(\left[a,b\right];\set{a,b})\to(X_R,\partial
    X_R)$ does not represent the relative homotopy group
    $\pi_1(X_R,\partial X_R)$

  A continuous map $\gamma\colon\mathrm{S}^1\to X_G$ is
  \term{cyclically reduced} if all edge-path restrictions of $\gamma$
  to subintervals $\mathrm{I}\subset\mathrm{S}^1$ are reduced edge
  paths.
\end{definition}

\par The following lemma is standard and follows easily from Stallings
folding \cite{stallings1,stallings0} and the definitions.

\begin{lemma}
   If $g\in G$ then there is a cyclically reduced edge path
   $\gamma\colon\mathrm{S}^1\to X_G$ representing the conjugacy class
   $\left[g\right]$.

   Let $\psi\colon X_G\immerses X_H$ be an immersion representing
   $G\into H$.  If $\gamma\colon\mathrm{S}^1\to X_G$ is a reduced edge
   path then $\psi\circ\gamma$ is a reduced edge path in $X_H$.
\end{lemma}

\par For each edge $E$ of $X_G,$ we introduced a subset $t_E$ of the
edge space $T_E\times I$.  We think of $t_E$ as a formal element
representing the path $I\to b_E\times\mathrm{I}$ with a fixed but
arbitrary orientation.  Let $\tau(t_E)$ be the image of the basepoint
of $T_E$ in the vertex space of $X_G$ at the terminal end of
$T_E\times\mathrm{I},$ and let $\iota(t_E)$ be the image of the
basepoint in the vertex space at the initial end of $T_E$.  Then every
nonelliptic element represented by a cyclically reduced path can be
thought of as a composition $t_E$'s, their inverses, and elements of
relative homotopy groups of vertex spaces.  Moreover, if the subword
$t_Egt^{-1}_E$ appears then $g$ is not contained in the image of $E$.

\par Let $\gamma\in G$ be represented by a cyclically reduced edge
path $\gamma$; $\psi\circ\gamma$ is an edge path in $X_H,$ and if it
is not cyclically reduced, then for some sub-path $t_E h t_E^{-1}$ of
$\gamma$ (we may need to reverse the orientation of $t_E$), the image
of this subpath is homotopic into $\eta_{\#}(T_E),$ which means that
$\left[h\right]\in\eta_{\#}(E)$. Since $\gamma$ is reduced,
$\left[h\right]\notin E,$ and since the image of $E$ in $\eta_{\#}(E)$
is finite index, for some $l>0,$ $\left[h\right]^l\in E$.  Since edge
groups are primitive unless adjacent to \qh\ vertex groups, $E$ must
be attached to a boundary component of a \qh\ vertex.  This implies
that $\eta_{\#}(E)$ is also attached to a boundary component of a
\qh\ vertex group, but this means $E\to\eta_{\#}(E)$ is an
isomorphism, contradicting the fact that $\left[h\right]\notin E$.



\par Let $G$ and $H$ be freely indecomposable limit groups, $H$
obtained from $G$ by adjoining roots to $\E,$ $\eta\colon
G\into H$.  Let $\Pi_G$ and $\Pi_H$ be principle cyclic decompositions
and suppose that if $K$ is elliptic in $\Pi_G$ if and only if
$\eta(K)$ is elliptic in $\Pi_H$.  Let $\psi\colon X_G\to X_H$ be an
immersion representing the inclusion.

\par Without loss of generality, suppose that all elements of
$\E$ are self-centralized and nonconjugate.  Let
$\E_e$ be the elements of $\E$ which are elliptic in
$\Pi_G$ and let $\E_h$ be the elements of $\E$ which
are hyperbolic in $\Pi_G$.

\par For each $E\in\E$ let $T_E$ be a torus representing $E,$
$T_{F(E)}$ a torus representing $F(E),$ and let $T_E\to T_{F(E)}$ be
the covering map corresponding to the inclusion $E\into F(E)$.  Let
$M_E$ be the mapping cylinder of the covering map.  If
$\group{\gamma}\in\E$ we abuse notation and refer to
$M_{\group{\gamma}}$ as $M_{\gamma}$.  The copy of $T_{F(E)}$ in
$M_E$ is the \term{core} of $M_E,$ and if $E$ is infinite cyclic, it
is the \term{core curve}.  The copy of $T_E$ in $M_E$ is the
\term{boundary}, and is denoted $\partial M_E$.

\par For each element $E\in\E_e,$ let $f_E\colon T_E\to X_G$
be a map representing the inclusion $E\into G$ which has image in a
vertex space of $X_G$.  If $E$ is an abelian vertex group of $\Pi_G$
then we identify $T_E$ with the torus $T_A\subset X_G$.  For each
$\group{\gamma}\in\E_h,$\footnote{All elements of
  $\E_h$ are infinite cyclic.} represent $\gamma$ by a
reduced edge path, abusing notation, $\gamma\colon\partial
M_{\gamma}\to X_G$.

\par Build a space $X'_G$ by attaching the $M_E$ and $M_{\gamma}$ to
$X_G$ along $T_E$ and $\img(\gamma)$ by the maps $f_E$ and $\gamma,$
respectively.

\par By hypothesis there is a $\pi_1$--surjective map $\psi'\colon
X'_G\to X_H$.  We choose this map carefully: For $E\in\E_e,$
$F(E)$ has elliptic image in $\Pi_H$.  Choose a map $T_{F(E)}\to X_H$
with image contained in the appropriate vertex space of $X_H,$ and
extend the map across $M_E$ so that $M_E$ also has image contained in
the vertex space of $X_H$.  For
$\group{\gamma}\in\E_h,$ the core curve of $M_{\gamma}$
is a $k_{\gamma}$--th root of $\gamma$.  Choose a cyclically reduced
representative of $\sqrt[k_{\gamma}]{\gamma}\colon S^1\to X_H$ and let
the map on the core curve agree with this representative.

\par The restriction of $\psi',$ defined thus far, to the disjoint
union of $X_G$ and the core curves of the $M_{\gamma},$ is transverse
to the subsets $T_{\eta_{\#}(E)}\times\set{\frac{1}{2}}$.  Extend $\psi'$ to
$X_G$ so the composition $M_{\gamma}\into X'_G\xrightarrow{\psi'}X_H$
is transverse to all $T_{\eta_{\#}(E)}\times\frac{1}{2}$.  Let
$\Lambda$ be the preimage

\[
  \psi'^{-1}\left(\sqcup_{E\in\E(G)}\left(T_{\eta_{\#}(E)}\times\set{\frac{1}{2}}\right)\right)
\]

\par Suppose some component of $\Lambda$ is a circle which misses the
boundary and core of some $M_{\gamma}$.  By transversality this
component of $\Lambda$ is a one manifold without boundary, and is
therefore a circle.  If this circle bounds a disk then there is a map
homotopic $\psi',$ which agrees with $\psi'$ on the core curves and
$X_G$ such that the number of connected components of the preimage is
strictly lower.  If the circle doesn't bound a disk in $M_{\gamma}$
then it is boundary parallel.  If this is the case then $\gamma$ is
elliptic and we have a contradiction.

\par Fix a mapping cylinder $M_{\gamma}$ and consider the preimage of
$\Lambda$ under the map $M_{\gamma}\to X'_G$.  The preimage is a graph
all of whose vertices are contained in the core curve of $M_{\gamma}$
or in the boundary of $M_{\gamma}$.  If any component of the preimage
of $\Lambda$ doesn't connect the boundary of $M_{\gamma}$ to the core
curve, then it is an arc and there is an innermost such arc which can
be used to show that one of either $\gamma$ or $\gamma'$ is not
reduced.  Thus the preimages of arcs connect the core curve to the
boundary.


\par Let $b$ be a point of intersection of $\Lambda$ with the core
curve of $M_{\gamma}$.  There are $k_{\gamma}$ arcs, where
$k_{\gamma}$ is the degree of the root added to $\gamma,$
$s_1,\dotsc,s_{k_{\gamma}}$ (cyclically ordered by traversing
$\partial M_{\gamma}$) in $\Lambda$ connecting $b$ to $\partial
M_{\gamma}$.  Now consider the arcs as paths
$s_j\colon\left[0,1\right]\to M_{\gamma}$.  The composition
$p_{\gamma}\define s^{-1}_2s_1$ is a path in $M_{\gamma}$ from
$\partial M_{\gamma}$ to $\partial M_{\gamma}$.  Let $D_{\gamma}$ be
the sub-path of $\gamma$ obtained by traversing $\partial M_{\gamma}$
from $*\define s_{1}\cap\partial M_{\gamma}$ to $*_2\define
s_{2}\cap\partial M_{\gamma}$.  The path $D_{\gamma}p_{\gamma}$ is
homotopic, relative to $*,$ to $s_1^{-1}\sqrt[k_{\gamma}]{\gamma}s_1$.
In particular,
\[(D_{\gamma}p_{\gamma})^{k_{\gamma}}\simeq\gamma\] 

\par A possible neighborhood of a component of $\Lambda$ is
illustrated in Figure~\ref{trackneighborhood}.

\begin{figure}[ht]
\psfrag{Lambda}{$\Lambda$}
 \psfrag{XR}{$X_R$}
 \psfrag{bEtimes}{$t_E$}
 \psfrag{TEtimes}{$T_E\times\mathrm{I}$}
 \psfrag{b}{$b$}
 \psfrag{s}{$s_{b,j}$}
 \psfrag{Dgamma}{$D_{\gamma}$}
 \psfrag{s1}{$s_1$}
 \psfrag{s2}{$s_2$}
 \psfrag{=s2}{$=*_2$}
 \psfrag{=s}{$=*$}
 \centerline{
   \includegraphics[scale=0.8]{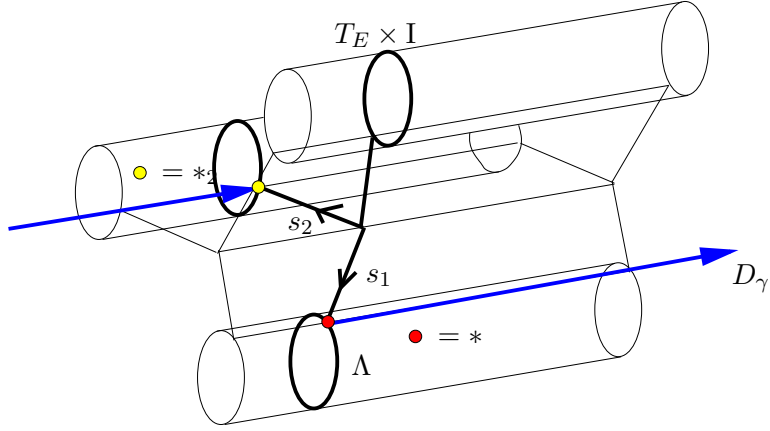}
 }
 \caption{A neighborhood of a component of $\Lambda$ in $X'_G$}
 \label{trackneighborhood}
 \end{figure}


\par Three interrelated lemmas.

\begin{lemma}
  \label{onetooneonedges}
  Suppose $\eta\colon G\into H,$ $H$ obtained from $G$ by adjoining
  roots to $\E,$ $G$ freely indecomposable.  Let $\Pi_H$ be a
  one-edged splitting of $G$ over a cyclic edge group $E_H$.  Let
  $\Pi_G$ be the splitting $G$ inherits from its action via $\eta$ on
  the Bass-Serre tree for $\Pi_H$.  Represent $\Pi_H$ by a graph of
  spaces $X_H,$ and choose a graph of spaces $X_G$ and an immersion
  $\psi\colon X_G\immerses X_H$ representing $\eta$.  Suppose that
  $\Pi_G$ is one-edged, and that the edge group is $E$.  If
  $\E_h$ is nonempty 
  then $E\into \eta_{\#}(E)$ is a
  \emph{proper} finite index inclusion.
\end{lemma}

\begin{proof}

  Let $\group{\gamma}\in\E_h,$ and represent $\gamma$
  by a reduced edge path crossing $t_E$.  Since $\psi$ is one-to-one on
  edge spaces, $p_{\gamma}$ is a closed path.  As such, it represents
  an element of the fundamental group of $X'_G$.  Then $\left[\psi\circ
    p_{\gamma}\right]\in\eta_{\#}(E)$.  If $\left[\psi\circ
    p_{\gamma}\right]\in\img(E)$ then there is a path $p'_{\gamma}$ in
  $T_E$ which is homotopic in $T_{\eta_{\#}(E)},$ relative to the
  image of $*,$ to $\psi\circ p_{\gamma}$.  Let
  $\alpha=D_{\gamma}p'_{\gamma}$.  Then $\psi\circ\alpha$ is homotopic
  rel the image of $*$ to $\psi\circ D_{\gamma}p_{\gamma}$.  But then
  $\left[\alpha\right]^{k_{\gamma}}=\gamma$ contradicting
  indivisibility of $\gamma$.
\end{proof}

\begin{lemma}
  \label{abeliansonto}
  Let $G\into G'$ be an adjunction of roots.  Let $\Pi_{G'}$ be a
  principle cyclic splitting of $G'$ with one abelian vertex group
  $A,$ let $\Pi_G$ be the associated splitting of $G,$ and represent
  $G\into G'$ by an immersion $\eta\colon X_G\immerses X_{G'}$
  reflecting $\Pi_G$ and $\Pi_{G'}$.  Suppose there is a unique vertex
  group $A'$ of $\Pi_G$ mapping to $A,$ and that there is at most one
  element of $\E$ conjugate into $A'$.  If $\eta$ is
  one-to-one on edges adjacent to $A'$ then the induced map $F(A')\to
  A/P(A)$ is onto.
\end{lemma}

\begin{proof}
  Let $E_1,\dotsc,E_n$ be the edges adjacent to $A',$ and set
  $F(E_i)=\eta_{\#}(E_i)=P(A)$.  Let $H$ be the limit group defined as
  follows: Let $\Delta=\Delta(R_j,E_i,A')$ be a graphs of groups
  representation of $\Pi_G$.  Let $\E_{R_j}$ be the
  subcollection of $\E$ consisting of elements conjugate into
  $R_j$.  Let 
  \[
    S_j\define\img_{G'}\left(\group{\adjoinroot{R_j}{\E_{R_j}}{},gF(E_i)g^{-1}}_{gE_ig^-1<R_j}\right)
    \]
  and \[A''\define\img_{G'}(F(A'),P(A))\] Let
  $H\define\Delta(S_j,F(E_i),A'')$ There are maps $G\into H\into
  G'$.  We now show that $H\into G'$ is actually surjective.  To do this
  we need to show that every element $\group{\gamma}$ of
  $\E_h$ has a $k_{\gamma}$--th root in $H$.  This is
  precisely the argument given at the end of
  Lemma~\ref{onetooneonedges}.  Let $G'\onto A/P(A)$ be the map which
  kills all vertex, edge groups, and stable letters, other than
  $A$.  The quotient map clearly kills everything except $A$ and
  $F(A'),$ giving the desired surjection.
\end{proof}


\begin{lemma}
  \label{noextensions}
  \label{vertexgroupsobtainedbyadjoiningroots}
  Let $(G,H,G')$ be a flight without any proper
  extensions.  Suppose $G$ is freely indecomposable, has no \qh\ vertex
  groups, and $\jcomp(G)=\jcomp(G')$.  Represent the $G\into H$ by an
  immersion $X_G\immerses X_H,$ representing
  $\jsj_H(G),$ and $\rjsj(H),$ respectively.  Then $\nu$ is one-to-one
  on edge spaces.

  Every vertex group $W$ of $H$ is obtained from a vertex group $V$ of
  $\jsj_H(G)$ by adjoining roots to the elements of $\E$
  which are conjugate into $V,$ along with edge groups incident to
  $V$.
\end{lemma}

\begin{proof}[Proof of Lemma~\ref{noextensions}] 
  Represent $G\into H$ by an immersion $X_G\to X_H,$ where $X_G$
  represents $\jsj_H(G)$ and $X_H$ represents the principle cyclic
  \jsj\ of $H$.  For each edge $E_i$ of $X_G$ let $e_i$ be a generator,
  let $k_i$ be the largest degree of a root of $e_i$ in $H,$ let
  $F(E)=\group{f_i},$ and let $E\into F(E)$ be the map which
  sends $e_i$ to $f_i^{k_i}$.  Let $\E'$ be the collection of
  elements of $\E$ which are elliptic in $H$ along with all
  edge groups of $\jsj_H(G)$.  

  Consider the group $\adjoinroot{G}{\E'}{}$.  Let
  $\Delta=\Delta(R_i,A_j,E_k)$ be a graph of groups representation of
  $\jsj_H(G)$.  Let $\E_{R_i}$ be the set of elements of
  $\E'$ which are conjugate into $R_i$.  Likewise, let
  $\E_{A_j}$ be the set of elements of $\E'$ which
  are conjugate into $A_j$.  Let 
  \[S_l=\img_H(\group{S_l,gBg^{-1}}_{gBg^{-1}<S_l,B\in\E_{S_l}})\]
  where $gBg^{-1}<S_l,$ and where $S_l$ is either some rigid vertex
  group $R_i$ or abelian vertex group $A_j$.  Let
  \[H'=\Delta(R'_l,A'_j,F(E_k))\] and choose a graph of spaces
  $X_{H'}$ representing this decomposition of $H'$.  There is a pair of
  maps of graphs of spaces $X_G\to X_{H'},$ $X_{H'}\to X_H,$ and there
  is an epimorphism $\adjoinroot{G}{\E'}{}\onto H'$.  The map
  $\psi'\colon X_{G}\to X_{H'}$ is one-to-one on edge
  spaces.  Moreover, $H'$ is a limit group since the map $H'\to H$ is
  clearly strict.

  The proof of the lemma will be complete if we can show that $\psi'$
  extends to $X'_G,$ that is, if $H'$ contains all roots of elements
  adjoined to $\E$.  Then the image of
  $\adjoinroot{G}{\E}{}$ (with the induced graph of groups
  decomposition) in $H'$ is a nontrivial extension of $H$.

  Consider the paths $D_{\gamma}$ and $p_{\gamma}$ defined previously
  through resolving.  We defined $p_{\gamma}\define s_2^{-1}s_1$ and
  set $*=s_1\cap\partial M_{\gamma}$.  Let $*_2\define s_2\cap\partial
  M_{\gamma}$.  To show that $H'$ has a $k_{\gamma}$--th root of
  $\gamma$ we need to show that $X_{H'}$ has a path $p'_{\gamma}$ from
  the image of $*_2$ to the image of $*$ whose image under $X_{H'}\to
  X_{G'}$ is homotopic rel endpoints to the image of $p_{\gamma}$.

  Suppose that $T_{E_1}\times\frac{1}{2}$ and
  $T_{E_2}\times\frac{1}{2}$ are the midpoints of edge spaces
  containing $*$ and $*_2,$ respectively, and suppose, without loss of
  generality, that $D_{\gamma}$ starts and ends by traversing the
  second and first halves of $T_{E_1}$ and $T_{E_2},$ respectively, in
  the positive direction.  The first key observation to make is that we
  can choose the orientations of $t_{E_i}$ so that the terminal
  endpoints of $t_{E_1}$ and $t_{E_2}$ are both contained in some
  $T_A$: $E_1$ and $E_2$ are conjugate in $H,$ must therefore be
  conjugate in $G$ since $\jcomp(G')=\jcomp(G),$ and cannot both be
  adjacent to a rigid vertex group of $G,$ otherwise there is a rigid
  vertex group $R$ of $G$ such that $v(R)>v(\eta_{\#}(R))$.  The only
  other possibility is that they are both adjacent to an abelian
  vertex group $A,$ as claimed.

  Let $t^+_{\varphi_{\#}(E_i)}$ be the half of $t_{\varphi_{\#}(E_i)}$
  obtained by traversing $t_{\varphi_{\#}(E_i)}$ from the midpoint to
  the terminal endpoint.  By Lemma~\ref{abeliansonto}, $H'\to H$ is
  surjective on abelian vertex groups, and by construction, the
  terminal endpoints of $t^+_{\varphi_{\#}(E_i)}$ agree.  Let
  $p''_{\gamma}\define t^+_{\varphi_{\#}(E_2)}
  (t^+_{\varphi_{\#}(E_1)})^{-1}$.  Then $p''_{\gamma}$ is a path from
  $\varphi(*_2)$ to $\varphi(*)$ whose image in $X_H$ is homotopic rel
  endpoints into $\psi_{\#}(E_1)(=\psi_{\#}(E_2))$.  Since $H'\onto H$
  is surjective on edge groups, there is a closed path $h_{\gamma}$ in
  $(T_{\varphi_{\#}(E_1)},\varphi(*))$ which maps to the image of
  $p_{\gamma}$.  Set $p'_{\gamma}\define h_{\gamma}p''_{\gamma}$.  The
  image of $p'_{\gamma}$ is homotopic rel endpoints to the image of
  $p_{\gamma}$ in $X_H$.  Arguing as in Lemma~\ref{onetooneonedges},
  $(\varphi\circ D_{\gamma})p'_{\gamma}$ is a $k_{\gamma}$--th root of
  $\varphi\circ\gamma$ and the map $X'_G\to X_H$ factors through
  $X_{H'}$.  

  Thus there is a map $X'_G\to X_{H'}$.  Since $H$ has no proper
  extensions,
  $\img_{H'}\left(\adjoinroot{G}{\E}{}\right)\to H$ is
  an isomorphism.  In particular, $X_G\to X_H$ is one-to-one on edges
  and the situation above never occurs.

  Consider the construction of $H'$.  Now that we know that $H'\cong
  H,$ $\Delta$ \emph{must} be the principle cyclic \jsj\ of $H$.  If
  there is a principle cyclic splitting of $H$ not visible in $\Delta$
  then it must be a cyclic splitting inherited from the relative (to
  incident edge groups) principle cyclic \jsj\ decomposition of some
  vertex group of $\Delta$.  On the other hand, all vertex groups of
  $\Delta$ must be elliptic in the principle cyclic \jsj\ of $H$ since
  they are obtained by adjoining roots to subgroups of $G$ which are
  guaranteed to be elliptic in the principle cyclic \jsj\ of $H$.
\end{proof}

\par This nearly completes the proof of Lemma~\ref{mainlemma}.  We
need to prove that the vertex groups of $\jsj^*_H(G')$ are obtained
from the images of the vertex groups of $\jsj(H)$ by adjoining roots,
and that $\pi$ is injective if its restrictions to vertex groups are
injective.

\par Let $\Delta=\Delta(R_i,A_j,E_k)$ be a graph of groups
decomposition representing the principle cyclic \jsj\ of $H$.  We know
that all vertex and edge groups of $\Delta$ map to vertex and edge
groups of $\jsj^*_H(G')$.  Let $\Phi_s(\pi)\colon \Phi_s(H)\onto G'$
be the strict homomorphism constructed
in~\cite[\S~\ref{STABLE-sec:constructstrict}]{louder::stable}, and
also in the appendix of this paper, and let $\Phi_s(\Delta)$ be the
principle cyclic decomposition of $\Phi_s(H)$ in which all images of
vertex groups of $\Delta$ are elliptic.  Clearly $\Phi_s(\pi)$ maps
elliptic subgroups of $\Phi_s(\Delta)$ to elliptic subgroups of
$\jsj^*_H(G')$.  Moreover, if $A$ is a noncyclic abelian vertex group
of $H,$ then by construction, $\Phi_s(\pi)$ maps $A/P(A)$ onto
$\pi_{\#}(A)/P(\pi_{\#}(A))$.  Thus all modular automorphisms of
$\Phi_s(H)$ supported on abelian vertex groups of $\Delta$ push
forward to modular automorphisms of $G'$.  Another consequence of the
hypothesis $\jcomp(G)=\jcomp(G')$ is that $\Phi_s(H)\to G'$ is one to
one on the set of edge groups adjacent to every vertex group, hence
every Dehn twist of $\Phi_s(H)$ pushes forward to a Dehn twist of
$G'$.  A strict map which allows all modular automorphisms to push
forward is an isomorphism, therefore $\Phi_s(\pi)$ is an isomorphism.

\par The third bullet follows immediately from the construction of
$\Phi_s$.

\appendix
\section{Constructing strict homomorphisms}
\label{appendix_strict}

\par We give here a description of the process of constructing strict
homomorphisms of limit groups.  Let $G$ be a group with a one-edged
splitting $\Delta$ with nonabelian vertex groups of the form $G\cong
R*_{\group{e}}S,$ and suppose there is a map $\varphi\colon G\to L,$
$L$ a limit group, which embeds $R$ and $S$.  Then $R$ and $S$ are
limit groups.  Suppose further that $\varphi$ embeds
$R*_{\group{e}}\cent_S(\group{e})$ and $\cent_R(\group{e})*_{\group{e}} S$.
Then $\varphi$ is \term{strict}, and $G$ is a limit group.  There is a
process, whose output is a limit group $\Phi_s(G),$ which takes the
data $(G,\Delta,L,\varphi)$ and produces a triple $G\to\Phi_S(G)\to
L,$ such that the composition is $\varphi,$ $\Phi_S(G)$ splits over
the centralizer of $\group{e},$ and $\Phi_S(G)\to L$ is strict.

\par The process is one of pulling centralizers and passing to images
of vertex groups in a systematic way.  The reader should compare this
to the more general construction detailed in~\cite{louder::stable},
and a formally identical version in the proof of
\cite[Lemma~7.9]{bf::lg}.  Let $G=G_0$. Define for

\begin{itemize}
  \item odd $i$: $G_i=R_{i-1}*_{\cent_{R_{i-1}}(\group{e})}S_i,$ where
    \[S_i\define\img_L(\cent_{R_{i-1}}(\group{e})*_{\cent_{S_{i-1}}(\group{e})}S_{i-1})\]
  \item even $i$: $G_i=R_i*_{\cent_{S_{i-1}}(\group{ e})}S_{i-1},$ where
    \[R_i\define\img_L(R_{i-1}*_{\cent_{S_{i-1}}(\group{ e})}\cent_{S_i}(\group{ e}))\]
\end{itemize}

\par We claim that this process terminates in finite time.  The
sequence of quotients $G_0\onto G_1\onto\dotsc$ embeds edge groups at
every step.  Since abelian subgroups of limit groups are finitely
generated and free, and since finitely generated free abelian groups
satisfy the ascending chain condition the assertion holds.  The direct
limit $G_{\infty}$ is called $\Phi_s(G)$.

\par This discussion is relevant to the proof of
Lemma~\ref{noextensions}, but we must vary the construction a little.
Let $\bar{H}$ be the quotient of $H$ obtained by passing to the images
in $G'$ of vertex groups of the (restricted) principle cyclic \jsj\ of
$H,$ with the induced graph of groups decomposition
$\Delta(\bar{R_i},A_j,E_k)$.  The \term{core} of $\bar{H},$
$\core(\bar{H})$ is the group obtained by replacing each abelian
vertex group $A$ by its peripheral subgroup.  Consider the situation
in Lemma~\ref{noextensions}.  There is a homomorphism
$\core(\bar{H})\to G',$ and each group is equipped with a principle
cyclic decomposition $\Delta_{\core(\bar{H})}$ and $\Delta_{G'},$
respectively.  Moreover, the nonabelian vertex groups of
$\Delta_{\core(\bar{H})}$ map to nonabelian vertex groups of $G',$ and
the edge groups of $\core(\bar{H})$ map to edge groups of
$\Delta_{G'}$.  The centralizers of edges incident to nonabelian
vertex groups of $G'$ are infinite cyclic, and this implies that in
the process of pulling centralizers in $\core(\bar{H})_i,$ the pulled
group is always infinite cyclic. Each vertex group of
$\core(\bar{H})_i$ has elliptic image in $G',$ and since $G'$ is
principle, centralizers are cyclic in the relevant vertex groups of
$G'$.  Iteratively adjoining roots to an infinite cyclic subgroup and
passing to quotients multiple times can be accomplished in one step,
thus the vertex groups of $\core(\bar{H})_{\infty}$ are obtained from
the vertex groups of $\core(\bar{H})$ by adjoining roots to incident
edge groups.  There are surjective maps $H\onto
\Phi_s(H)\define\core(\bar{H})_{\infty}*_{Z(P(A_j))}(Z(P(A_j))\oplus
A/P(A_j))\onto L$.

\bibliographystyle{amsalpha} \bibliography{krull}

\end{document}